\documentclass[12pt,a4paper]{amsart}

\makeatletter
\renewcommand\normalsize{%
    \@setfontsize\normalsize{11.7}{14pt plus .3pt minus .3pt}%
    \abovedisplayskip 10\p@ \@plus4\p@ \@minus4\p@
    \abovedisplayshortskip 6\p@ \@plus2\p@
    \belowdisplayshortskip 6\p@ \@plus2\p@
    \belowdisplayskip \abovedisplayskip}
\renewcommand\small{%
    \@setfontsize\small{9.5}{12\p@ plus .2\p@ minus .2\p@}%
    \abovedisplayskip 8.5\p@ \@plus4\p@ \@minus1\p@
    \belowdisplayskip \abovedisplayskip
    \abovedisplayshortskip \abovedisplayskip
    \belowdisplayshortskip \abovedisplayskip}
\renewcommand\footnotesize{%
    \@setfontsize\footnotesize{8.5}{9.25\p@ plus .1pt minus .1pt}
    \abovedisplayskip 6\p@ \@plus4\p@ \@minus1\p@
    \belowdisplayskip \abovedisplayskip
    \abovedisplayshortskip \abovedisplayskip
    \belowdisplayshortskip \abovedisplayskip}
\ifdefined\pdfpagewidth
\else
\fi
\calclayout
\makeatother

\usepackage{graphicx}%
\usepackage{multirow}%
\usepackage{amsmath,amssymb,amsfonts}%
\usepackage{amsthm}%
\usepackage{mathrsfs}%
\usepackage[title]{appendix}%
\usepackage{xcolor}%
\usepackage{textcomp}%
\usepackage{manyfoot}%
\usepackage{booktabs}%
\usepackage{algorithm}%
\usepackage{algorithmicx}%
\usepackage{algpseudocode}%
\usepackage{listings}%
\usepackage{url}
\usepackage[font=small]{caption}

\theoremstyle{plain}%
\newtheorem{theorem}{Theorem}
\newtheorem{proposition}[theorem]{Proposition}%
\newtheorem{lemma}[theorem]{Lemma}%
\newtheorem{corollary}[theorem]{Corollary}%
\newtheorem{claim}{Claim}%

\theoremstyle{remark}%
\newtheorem{remark}{Remark}%
\newtheorem{example}{Example}%

\theoremstyle{definition}%
\newtheorem{definition}{Definition}%

\begin{document}
\title[The mathmatics behind the model] {A novel mathematical model of protein interactions (2): The mathematics behind the model}

\author {Naoto Morikawa}

\begin{abstract}
This article is a sequel to “A novel mathematical model of protein interactions from the perspective of electron delocalization (2025).”, where protein molecules are modelded as loops of triangles with no singular holes inside. Here we consider two questions: (1) How can we define the shape of a molecule? (i.e., how can we define the shape of a loop with no singular holes inside?), and (2) Is the given loop a molecule? (i.e., are the holes within a loop are singular?).

These questions are answered by (1) giving a defining system of equations for the shape of a molecule using concepts from category theory, and (2) giving a equation that determines whether holes within a loop are singular using concepts from cohomology theory, respectively.

No prior knowledge of the previous paper, category theory, or cohomology theory is required. 
\end{abstract}


\maketitle

\tableofcontents

\section{Introduction}\label{sec1}
This article is a sequel to “A novel mathematical model of protein interactions from the perspective of electron delocalization (2025)” \cite{NM1}. Here I will describe two research topics concerning the model. Since this article is self-contained, there is no need to read the previous paper,.

\begin{figure}
\centering
\captionsetup{width=1.0\linewidth}
\includegraphics{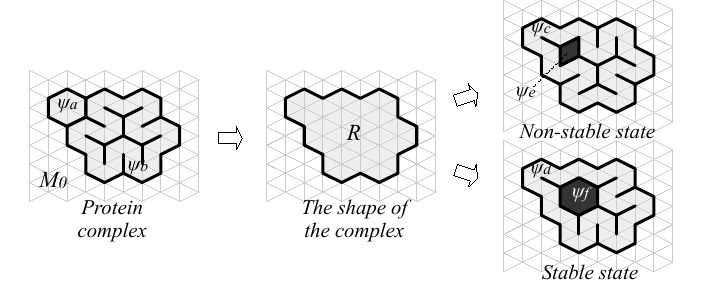}
\caption{The model of protein interactions.}
\label{figure1}
\end{figure}

In \cite{NM1}, protein molecules are modeled as loops of triangles with no singular holes inside. A hole inside is called \textit{singular} if there is no decomposition of the hole into loops. By definition, protein molecules interact if they can be fused into a loop with no singular holes (i.e., a molecule).

In Figure \ref{figure1} left, two molecules $\psi_a$ and $\psi_b$ are given on a triangular mesh $M_0$. An interaction between them occurs if there is a molecule with the same shape as $R$ (Figure \ref{figure1} middle).

In Figure \ref{figure1} top right, loop $\psi_c$ has a hole inside. Since the hole has no loop decomposition, $\psi_c$ does not yield a molecule with the same shape as $R$ (an ``non-stable'' state). 

In Figure \ref{figure1} bottom right, loop $\psi_d$ has a hole inside. In this case, since the hole is a loop $\psi_f$ of length $6$, the two loops $\psi_d$ and $\psi_f$ provides a molecule with the same shape as $R$ (a ``stable'' state). Therefore, $\psi_a$ and $\psi_b$ interact. 

This article addresses two questions we face in the computation of the interaction between molecules:
\begin{enumerate}
\item How can we define the shape of a (protein) molecule? (.i.e, how can we define the shape of a loop with no singular holes inside?), 
\item Is the given loop a (protein) molecule? (i.e., are the holes within a loop singular?).
\end{enumerate}

The former is answered by giving a defining system of equations for the shape of a molecule using concepts from category theory (Definition \ref{def_shape1}, \ref{def_shape2}, \ref{def_shape3}, and \ref{def_shape4}). The latter is answered by giving a determining equation for the question using concepts from cohomology theory (Proposition \ref{prop_defeq}). No prior knowledge of category theory \cite{CAT1}, or cohomology theory \cite{COH1, COH2} is required. A basic understanding of differential geometry should suffice. 
 
Finally, Genocript (\url{http://www.genocript.com}) is the one-man bio-venture started by Naoto Morikawa in $2000$.

\section{Basic concepts}\label{sec2}

In the following, $\mathbf{Z}$ denotes the collection of all integers, $\mathbf{N}$ denotes the collection of all natural numbers ($0 \not\in \mathbf{N}$), and $\mathbf{R}^n$ denotes the $n$-dimensional Euclidean space ($n \in \mathbf{N}$).

Loops of triangles are defined on a triangular mesh.

\begin{definition} [Mesh $M_0$]\label{def_mesh}
A \textit{triangular mesh} $M$ is a triplet 
\begin{equation*}
(Tri(M),\ Edge(M),\ Ver(M)), 
\end{equation*}
where 
\begin{align*}
&Tri(M) \text{ is the collection of all triangles of $M$},\\
&Edge(M) \text{ is the collection of all edges of the triangles in $Tri(M)$}, \\ 
&Ver(M) \text{ is the collection of all vertices of the edges in $Edge(M)$}. 
\end{align*}
$M_0$ is the triangular mesh obtained by dividing $\mathbf{R}^2$ into equilateral triangles (Figure \ref{figure1} (a)). In particular, $M_0$ is defined in $\mathbf{R}^2$.
\end{definition}
\begin{remark}
In general, a triangular mesh $M$ is defined in $\mathbf{R}^n$ ($n \in \mathbf{N}$). 
\end{remark}

\begin{definition} [$M_1 \subset M_2$]\label{def_submesh}
Given two triangular meshes $S$ and $M$. $S$ is called a \textit{submesh} of $M$ if
\begin{align*} 
Tri(S) &\subset Tri(M), \\
Edge(S)&:=\{ ab \in Edge(M)\ |\ \exists c \in Ver(M) \text{ such that } abc \in Tri(S) \}, \\
Ver(S)&:=\{ a \in Ver(M) \ |\  \exists b \in Ver(M) \text{ such that } ab \in Edge(S) \}.
\end{align*}
$SUBMESH(M)$ denotes the collection of all submeshes of $M$. Given $M_1, M_2 \in SUBMESH(M)$. \underline{We write $M_1 \subset M_2$ if $M_1 \in SUBMESH(M_2)$}.
\end{definition}

\begin{remark}
For simplicity, we only consider submeshes of $M_0$.
\end{remark}

\begin{definition} [$M_1 \cap M_2$ and $M_1 \cup M_2$]
Given $M_1, M_2 \subset M_0$. The \textit{intersection} $M_1 \cap M_2$ of $M_1$ and $M_2$ is the submesh of $M_0$ defined by
\begin{align*} 
Tri(M_1 \cap M_2)&:=Tri(M_1)\cap Tri(M_2), \\
Edge(M_1 \cap M_2)&:=\{ ab \in Edge(M_1) \  |\  \exists c \in Ver(M_1) \text{ such that } abc \in Tri(M_1 \cap M_2) \}, \\
Ver(M_1 \cap M_2)&:=\{ a \in Ver(M_1) \  |\  \exists b \in Ver(M_1) \text{ such that } ab \in Edge(M_1 \cap M_2) \}.
\end{align*}
The \textit{union} $M_1 \cup M_2$ of $M_1$ and $M_2$ is also defined in the same way, i.e.,
\begin{equation*}
Tri(M_1 \cup M_2):=Tri(M_1)\cup Tri(M_2).
\end{equation*}
By definition,
\begin{equation*}
M_1 \cap M_2 \subset M_i \subset M_1 \cup M_2 \quad (i=1,2).
\end{equation*}
\end{definition}

\begin{definition}[A trajectory of triangles]\label{def_traj}
Given $M \subset M_0$. A \textit{trajectory} on $M$ is an arrangement of triangles in $Tri(M)$ obtained by connecting triangles via common edges. In general, arrangements are non-linear (i.e., they contain branches). $TRAJ(M)$ denotes the collection of all trajectories on $M$, i.e,
\begin{equation*}
TRAJ(M):=\{ \psi \subset Tri(M) \ | \  \text{$\psi$ is a trajectory on $M$}\}.
\end{equation*}
Given $\psi \in TRAJ(M)$ and $abc \in Tri(M)$. \underline{We write $abc \in \psi$ if $abc$ is contained} \underline{in $\psi$}. Given $\psi_1, \psi_2 \in TRAJ(M)$. \underline{We write $\psi_1 \subset \psi_2$ if $abc \in \psi_2$ for $\forall abc \in \psi_1$}.
\end{definition}

\begin{definition}[Branch, regular, terminal, and isolated triangles]
Given $M \subset M_0$, $\psi \in TRAJ(M)$, and $abc \in \psi$. $abc$ is called a \textit{branch} triangle if it is connected to three adjacent triangles, a \textit{regular} triangle if it is connected to two adjacent triangles, a \textit{terminal} triangle if it is connected to one adjacent triangle, and an \textit{isolated} triangle if it is not connected to any adjacent triangles.
\end{definition}

\begin{definition} [Normal edges]
Given $M \subset M_0$, $\psi \in TRAJ(M)$, and $abc \in \psi$. The \textit{normal edges} of $abc$ are the edges that are not used in the connection in $\psi$. A branch triangle has no normal edges. A regular triangle has one normal edge. A terminal triangle has two normal edges. An isolated triangle has three normal edges. 
\end{definition}

\begin{example} 
In Figure \ref{figure1} left, two close trajectories $\psi_a$ and $\psi_b$ are shown. Both $\psi_a$ and $\psi_b$ consist of regular triangles. The normal edges are shown as thick line segments.
\end{example}

\begin{example} 
In Figure \ref{figure1} top right, the hole within $\psi_c$ consists of two terminal triangles.
\end{example}

\begin{definition} [Region] \label{def_region}
Given $M \subset M_0$ and $R \subset \mathbf{R}^2$. $R$ is called a \textit{region} on $M$ if it is a union of triangles in $Tri(M)$, i.e.,
\begin{equation*}
\exists U \subset Tri(M) \text{ such that } R=\bigcup_{abc \in U}  abc  \subset \mathbf{R}^2.
\end{equation*}
$REG(M)$ denotes the collection of all regions on $M$, i.e.,
\begin{equation*}
REG(M):=\{ \bigcup_{abc \in U}  abc  \subset \mathbf{R} \ |\  U \subset Tri(M) \}.
\end{equation*}
 A collection of regions is called a \textit{region complex}. 
 Given a region complex $R=\{r_1,\ r_2,\ \ldots,\ r_k\} \subset REG(M)$ ($k \in \mathbf{N}$). $|R|$ denotes the union of all regions in $R$, i.e., 
\begin{equation*}
|R|:=r_1 \cup r_2 \cup \cdots \cup r_k \quad \in REG(M).
\end{equation*}
\end{definition}

\begin{definition} [$|\psi |_0$]
Given $M \subset M_0$ and $\psi \in TRAJ(M)$. The \textit{region $|\psi |_0$ associated with $\psi$} is the union of all triangles in $\psi$, i.e.,
\begin{equation*}
|\psi |_0:=\bigcup_{abc \in \psi}  abc \quad  \in REG(M).
\end{equation*}
The \textit{length} $len_0(\psi)$ of $\psi$ is the number of triangles in $\psi$, i.e.,
\begin{equation*}
 len_0(\psi):=\sharp\{abc \in Tri(M) \  |\ abc \in \psi \} \quad \in \mathbf{Z}.
\end{equation*}
\end{definition}

Note that $|\psi |_0$ may have holes inside, i.e., voids located inside. 

\begin{definition} [$HOLE(\psi)$]\label{def_hole}
Given $M \subset M_0$ and $\psi \in TRAJ(M)$. $HOLE(\psi)$ denotes the collection of all holes within $|\psi |_0$, i.e.,
\begin{equation*}
HOLE(\psi):=\{ r \in REG(M_0)\ |\ \text{$r$ is a hole within $|\psi|_0$} \} \quad \subset REG(M_0).
\end{equation*}
$HOLE(\psi)$ may contain not only regions on $M$ but also regions on $M_0$. 
\end{definition}

\begin{definition} [$|\psi |$]
Given $M \subset M_0$ and $\psi \in TRAJ(M)$. The \textit{extended region} $|\psi|$ is defined by
\begin{equation*}
|\psi |:=|\psi |_0 \cup \left(\bigcup_{r \in HOLE(\psi)} r\right) \quad \in  REG(M_0).
\end{equation*}
The \textit{extended length} $len(\psi)$ of $\psi$ is the number of triangles in $|\psi |$, i.e.,
\begin{equation*}
len(\psi):=\sharp\{abc \in Tri(M) \  |\ abc \subset |\psi| \} \quad \in \mathbf{Z}.
\end{equation*}
Note that $|\psi |$ has no holes inside.
\end{definition}

\begin{definition} [Loops] \label{def_loop}
Given $M \subset M_0$ and $\psi \in TRAJ(M)$. $\psi$ is called a \textit{loop} if it is a closed trajectory consisting only of regular triangles. $LOOP(M)$ denotes the collection of all loops on $M$, i.e.,
\begin{equation*}
LOOP(M):=\{ \psi \in TRAJ(M) \  |\  \text{$\psi$ is a loop on $M$}\} \quad \subset TRAJ(M).
\end{equation*}
 A collection of loops is called a \textit{loop complex}. $|LOOP|(M)$ denotes the collection of all extended regions of loops on $M$, i.e.,
\begin{equation*}
|LOOP|(M):=  \{|\psi| \in REG(M)\ |\ \psi \in LOOP(M)\} \quad \subset REG(M).
\end{equation*}
\end{definition}

\begin{example}
In Figure \ref{figure1} top right, $\psi_c$ is a closed trajectory of length $46$. Define $M \subset M_0$ by 
\begin{equation*}
Tri(M)=\{ abc \in Tri(M_0)\ |\  abc \in \psi_c \}. 
\end{equation*}
Then, 
\begin{align*}
&\psi_c \in LOOP(M), \\
&HOLE(\psi_c) =\{ |\psi _e|\}, \\
&|\psi_c |=|\psi_c |_0 \cup |\psi _e|.
\end{align*}
Moreover, 
\begin{align*}
&len_0(\psi_c)=46,\\
&len_0(\psi_e)=len(\psi_e)=2, \\
&len(\psi_c)=48.
\end{align*}
\end{example}

\begin{example}
In Figure \ref{figure1} bottom right,, $\psi_d$ is a closed trajectory of length $42$. Define $M \subset M_0$ by 
\begin{equation*}
Tri(M)=\{ abc \in Tri(M_0)\ |\  abc \in \psi_d \}.
\end{equation*}
Then, 
\begin{align*}
&\psi_d \in LOOP(M), \\
&HOLE(\psi_d) =\{ |\psi _f|\}. \\ 
&|\psi_d |=|\psi_d |_0 \cup |\psi _f|.
\end{align*}
Moreover,
\begin{align*}
&len_0(\psi_d)=42,\\
&len_0(\psi_f)=len(\psi_f)=6, \\
&len(\psi_d)=48.
\end{align*}
\end{example}

\begin{definition} [Loop decomposition of a region] \label{def_decomp}
Given $M \subset M_0$ and $R \in REG(M)$. $R$ is called \textit{loop decomposable} if it can be divided into loops, i.e., $\exists \psi_1,\ \psi_2,\ \cdots ,\ \psi_k \in LOOP(M)$ ($k \in \mathbf{N}$) such that
\begin{equation*}
R=|\psi_1|_0 \cup |\psi_2|_0 \cup \cdots \cup |\psi_k|_0.
\end{equation*}
\end{definition}

\begin{example}
Let's consider the region $R \subset \mathbf{R}^2$ in Figure \ref{figure1} middle. Define $M \subset M_0$ by 
\begin{equation*}
Tri(M)=\{ abc \in Tri(M_0)\ |\  abc \subset R \}.
\end{equation*}
Then, 
\begin{align*}
&\psi_d,\ \psi _f \in LOOP(M), \\ 
&R=|M|=|\psi_d|_0 \cup |\psi_f|_0
\end{align*}
(Figure \ref{figure1} bottom right). That is, $R$ is loop decomposable.
\end{example}

\begin{definition}[Locally regular loops]
Given $M \subset M_0$ and $\psi \in LOOP(M)$. $\psi$ is called \textit{locally regular (on $M_0$)} if all regions in $HOLE(\psi)$ are loop decomposable. 
\end{definition}

\begin{definition}[Regular loops]
Given $M \subset M_0$ and locally regular $\psi \in LOOP(M)$. $\psi$ is called \textit{regular (on $M_0$)} if the region outside $|\psi|$ (i.e., $\mathbf{R}^2 \setminus |\psi|$) can be decomposed into trajectories with no branches, terminals, or isolated triangles.
\end{definition}

\begin{definition}[Molecules] \label{def_mol}
Given $M \subset M_0$ and $\psi \in LOOP(M)$. $\psi$ is called a \textit{molecule (on $M$)} if it is locally regular and $len(\psi)$ is finite. $MOL(M)$ denotes the collection of all molecules on $M$, i.e.,
\begin{equation*}
MOL(M):=\{ \psi \in LOOP(M) \  |\  \text{$\psi$ is a molecule}\} \quad \subset LOOP(M).
\end{equation*}
 A collection of molecules is called a \textit{molecule complex}. $|MOL|(M)$ denotes the collection of all extended regions of molecules on $M$, i.e.,
\begin{equation*}
|MOL|(M):= \{|{\psi}| \in REG(M) \ |\  \psi \in MOL(M)\} \quad \subset REG(M).
\end{equation*}
\end{definition}

\begin{definition} [Interaction between molecules]
Given $M \subset M_0$ and $\psi_1, \psi_2 \in MOL(M)$. $\psi_1$ and $\psi_2$ \textit{interact} if $\exists \psi_3 \in MOL(M)$ such that
\begin{equation*}
|\psi_3|=|\psi_1| \cup |\psi_2| \quad  \subset \mathbf{R}^2.
\end{equation*}
\end{definition}

\begin{lemma}[A condition for interaction]\label{INT_and_LD}
Given $M \subset M_0$, $\psi_1, \psi_2 \in MOL(M)$, and $\psi_3 \in LOOP(M)$. Suppose $|\psi_3|=|\psi_1| \cup |\psi_2|$. Then, $\psi_1$ and $\psi_2$ interact if all regions in $HOLE(\psi_3)$ are loop decomposable.
\end{lemma}
\begin{proof}
It follows immediately from the definitions.
\end{proof}

\begin{remark}
Note that $REG(M) \not\subset |LOOP|(M)$ ($M \subset M_0$). That is, for some $\psi_1, \psi_2 \in MOL(M)$, 
there exists no $\psi_3 \in LOOP(M)$ such that $|\psi_3|=|\psi_1| \cup |\psi_2|$.
\end{remark}

In the following, 
\begin{enumerate}
\item Additive notation ($+$) is used to denote \underline{fusion of loops}, 
\item Set theory notation ($\cup$) is used to denote \underline{union of regions},
\item Multiplicative notation (juxtaposition) is used to denote loop complices and region complices. 
\end{enumerate}
For example, in Figure \ref{figure1}, 
\begin{align*}
\psi_d&= \psi_a + \psi_b  \quad \in LOOP(M_0), \\
|\psi_d|&= |\psi_a| \cup |\psi_b| \quad \in REG(M_0), \\
{\psi_a}{\psi_b}&=\text{``the loop complex of Figure \ref{figure1} left''} \quad \subset LOOP(M_0), \\
|\psi_a||\psi_b|&=\text{``the corresponding region complex''} \quad   \subset REG(M_0).
\end{align*}
Note that 
\begin{align*}
|\psi_a||\psi_b| &\neq  |\psi_a| \cup |\psi_b|, \\
||\psi_a||\psi_b|| &= |\psi_a\psi_b| =  |\psi_a| \cup |\psi_b|, 
\end{align*}
where  $|\psi_a||\psi_b|$ is a region complex and $ |\psi_a| \cup |\psi_b|$ is a union of regions.

\section{Research topics} \label{sec3}

\subsection{How can we define the shape of a molecule?}\label{sec31}\ 

Using concepts from category theory, we derive a system of equations that specifies the shape of a given molecule. The idea is to describe molecules using ``integral'' molecules, just as rational numbers are written using integers. The question is, ``To what extent can a molecule be specified by integral molecules?''

\begin{figure}
\centering
\captionsetup{width=1.0\linewidth}
\includegraphics{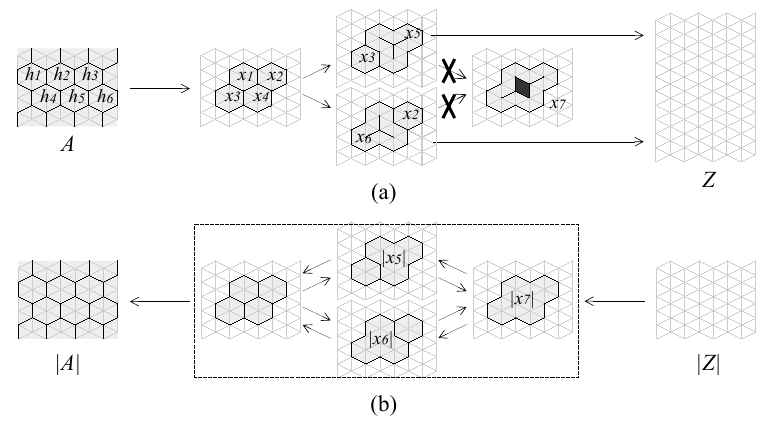}
\caption{(a) The binary relation $\rightarrow_M$ between molecule complices of $Obj(Mol)$. (b) The binary relation $\rightarrow$ between region complices of $Obj(Reg)$.}
\label{figure2}
\end{figure}

\subsubsection{Loops and regions}\label{sec311}\ 

A hexagon consisting of $6$ triangles is the shortest loop on $M_0$ (Definition \ref{def_mesh}), i.e., a loop of length $6$. Since loops of length $6$ have no holes, they are molecules. In the following, \underline{hexagons play the role which $1$ play for $\mathbf{Z}$}. That is, ''integral'' molecules are obtained by fusion of hexagons, and ''rational'' molecules are obtained by fission of ''integral'' molecules.

First, mesh $M_0$ is divided into a collection of hexagons:
\begin{definition} [$A$ and $|A|$]
``$A$'' denotes a collection of non-overlapping loops of length $6$ on $M_0$ that fill $|M_0|$ (Figure \ref{figure2} (a) left), i.e.,
\begin{equation*}
A:=\{o_i \in MOL(M_0) \ | \ i \in \mathbf{Z} \} \subset MOL(M_0)
\end{equation*}
such that
\begin{enumerate}
\item $o_i$ is a loop of length $6$ ($i \in \mathbf{Z}$),
\item $o_i$ and $o_j$ share no triangles if $o_i \neq o_j$ ($i,\ j \in \mathbf{Z}$),
\item $\bigcup_{i \in \mathbf{Z}} o_i=|M_0|$.
\end{enumerate}
$|A|$ denotes the collection of the corresponding hexagons, i.e.,
\begin{equation*}
|A|:=\{|o| \in REG(M_0)\ |\  o \in A\} \quad \subset REG(M_0).
\end{equation*}
\end{definition}

\begin{remark}
Each loop $o_i$ of $A$ is placed at a specific position on $M_0$ depending on the value of $i$. We use name tags $h_i$ and $x_i$ ($i \in \mathbf{Z}$) to represent the loops of $A$. That is, if we write $h_j \in A$, then $\exists j \in \mathbf{Z}$ such that $h_i=o_j$. For example, in Figure \ref{figure2} (a) left, $h_i$'s represent the loops of $A$ at the corresponding positions. $|h_i|$'s represents the hexagons at the corresponding positions.
\end{remark}

Next, we define binary relations (``arrows'') on three types of objects. ``Arrows'' represent the ``fusion'' of the specified object.
\begin{definition}[$Obj(Mol)$ and $\rightarrow_M$]
The objects $Obj(Mol)$ and the binary relation $\rightarrow_M$ on them are defined as follows:
\begin{enumerate}
\item $Obj(Mol)$ denotes the collection of all subsets of $MOL(M_0)$ (i.e., the \textit{power set} of $MOL(M_0)$):
\begin{equation*}
Obj(Mol):=Power(MOL(M_0)).
\end{equation*}
Elements of $Obj(Mol)$ are called \textit{molecule complices} (on $M_0$). 

Given $m=\{\psi_i\} \in Obj(Mol)$. Then, $m$ is expressed as a monomial, i.e.,
\begin{equation*}
m=\prod_{i}{\psi}_i  \in Obj(Mol),
\end{equation*}
using the multiplicative notation. The \textit{region $|m|$ of $m$} is defined by
\begin{equation*}
|m|:= \bigcup_{{\psi}_i \in m}\ |\ {\psi}_i|  \in REG(M_0).
\end{equation*}
``$Z$'' denotes the empty set of $Obj(Mol)$ (Figure \ref{figure2} (a) right), i.e., 
\begin{equation*}
Z:= \emptyset  \in Obj(Mol).
\end{equation*}

\item Given $m_1=\prod_ i {\psi}_i,\  m_2=\prod_i {\phi}_i \in Obj(Mol)$. The binary relation $\rightarrow_M$ between $m_1$ and $m_2$ is defined by 
\begin{multline*} 
\phantom{aaaaaaaa}m_1 \rightarrow_M m_2 \text{ if and only if} \\
\text{ for } \forall {\phi}_i \in m_2,\  \exists B_{\phi_i}  \subset m_1 \text{ such that } {\phi}_i =\sum_{\psi_j \in B_{\phi_i}} {\psi}_j.
\end{multline*}
\end{enumerate}
\end{definition}

\begin{remark}
$Obj(Mol)$ and $\rightarrow_M $ define the category $\mathbf{Mol}$ of molecule complices.
\end{remark}

\begin{definition}[$Obj(Loop)$ and $\rightarrow_L$]
The objects $Obj(Loop)$ and the binary relation $\rightarrow_L$ on them are defined as follows:
\begin{enumerate}
\item $Obj(Loop)$ denotes the power set of $LOOP(M_0)$, i.e.,
\begin{equation*}
Obj(Loop):=Power(LOOP(M_0)).
\end{equation*}
Elements of $Obj(Loop)$ are called \textit{loop complices} (on $M_0$). 

Given $l =\{\psi_i\} \in Obj(Loop)$. Then, $l$ is expressed as a monomial, i.e.,
\begin{equation*}
l=\prod_{i}{\psi}_i  \in Obj(Loop),
\end{equation*}
using the multiplicative notation. The \textit{region $|l|$ of $l$} is defined by
\begin{equation*}
|l|:= \bigcup_{{\psi}_i \in l} |{\psi}_i|  \in REG(M_0).
\end{equation*}
``$Z$'' denotes the empty set of $Obj(Loop)$, i.e., 
\begin{equation*}
Z:= \emptyset  \in Obj(Loop).
\end{equation*}

\item Given $l_1=\prod_ i {\psi}_i,\  l_2=\prod_i {\phi}_i \in Obj(Loop)$. The binary relation $\rightarrow_L$  between $l_1$ and $l_2$ is defined by 
\begin{multline*} 
\phantom{aaaaaaaa}l_1 \rightarrow_L l_2 \text{ if and only if} \\
\text{ for } \forall {\phi}_i \in l_2,\  \exists B_{\phi_i}  \subset l_1 \text{ such that } 
|{\phi}_i| =\bigcup_{\psi_j \in B_{\phi_i}} |{\psi}_j|.
\end{multline*}
\end{enumerate}
\end{definition}

\begin{remark}
$Obj(Loop)$ and $\rightarrow_L $ define the category $\mathbf{Loop}$ of loop complices.
\end{remark}

\begin{definition}[$Obj(Reg)$ and $\rightarrow$]
The objects $Obj(Reg)$ and the binary relation $\rightarrow$ on them are defined as follows:
\begin{enumerate}
\item $Obj(Reg)$ denotes the power set of $REG(M_0)$, i.e.,
\begin{equation*} 
Obj(Reg):=Power(REG(M_0)).
\end{equation*}
Elements of $Obj(Reg)$ are called \textit{region complices} (on $M_0$). 

Given $s=\{r_i\} \in Obj(Reg)$. Then, $s$ is expressed as a monomial, i.e.,
\begin{equation*}
s=\prod_{i}r_i  \in Obj(Reg),
\end{equation*}
using the multiplicative notation. The \textit{region $|s|$ of $s$} is defined by
\begin{equation*}
|s|:= \bigcup_{r_i \in s} r_i  \in REG(M_0).
\end{equation*}
$|Z|$ denotes the empty set of $Obj(Reg)$ (Figure \ref{figure2} (b) right), i.e., 
\begin{equation*}
|Z|:= \emptyset  \in Obj(Reg).
\end{equation*}

\item Given $s_1={\prod}_ i r_{1i},\  s_2={\prod}_j r_{2j} \in Obj(Reg)$. The binary relation $\rightarrow$ between $s_1$ and $s_2$ is defined by 
\begin{equation*} 
\phantom{aaaaaaaa}s_1 \rightarrow s_2 \text{ if and only if } \bigcup_{r_{1i} \in s_1} r_{1i} \subset \bigcup_{r_{2j} \in s_2} r_{2j}.
\end{equation*}
\end{enumerate}
That is, it is an inclusion relationship.
\end{definition}

\begin{remark}
$Obj(Reg)$ and $\rightarrow$ define the category $\mathbf{Reg}$ of region complices.
\end{remark}

\begin{lemma}\label{categories} Some of the basic properties are summarized below.
\begin{enumerate}
\item $A,\ Z \in Obj(Mol) \subset Obj(Loop)$
\item $|A|,\ |Z| \in Obj(Reg)$.  
\item $l \rightarrow_L Z$ for $\forall l \in Obj(Loop)$.
\item $m \rightarrow_M Z$ for $\forall m \in Obj(Mol)$.
\item $|Z| \rightarrow s \rightarrow |A|$ for $\forall s \in Obj(Reg)$.
\end{enumerate}
\end{lemma}

\begin{example}
In Figure \ref{figure2} (a) middle, we have
\begin{align*}
A &\rightarrow_M x_1x_2x_3x_4 \rightarrow_M  x_3x_5 \rightarrow_M Z,\\
A &\rightarrow_M x_1x_2x_3x_4 \rightarrow_M  x_2x_6 \rightarrow_M Z,\\
A &\rightarrow_L x_1x_2x_3x_4 \  \rightarrow_L x_3x_5 \  \rightarrow_L x_7 \ \rightarrow_L Z, \\
A &\rightarrow_L x_1x_2x_3x_4 \  \rightarrow_L x_2x_6 \  \rightarrow_L x_7 \ \rightarrow_L Z
\end{align*}
for molecule complices $x_1x_2x_3x_4$, $x_3x_5$, $x_2x_6$, and a loop $x_7$.
Note that 
\begin{equation*}
|x_1x_2x_3x_4|= |x_3x_5|=|x_2x_6|=|x_7|.
\end{equation*}
\end{example}

\begin{example}
In Figure \ref{figure2} (b) middle, we have
\begin{align*}
|A| &\leftarrow |x_1||x_2||x_3||x_4| \leftrightarrows |x_3||x_5| \leftrightarrows |x_7| \leftarrow |Z|, \\
|A| &\leftarrow |x_1||x_2||x_3||x_4| \leftrightarrows |x_2||x_6| \leftrightarrows |x_7| \leftarrow |Z|
\end{align*} 
for region complices $|x_1||x_2||x_3||x_4|$, $|x_3||x_5|$, $|x_2||x_6|$, and $|x_7|$.
\end{example}

\subsubsection{The defining system of equations for integral molecules }\label{sec312}\ 

We use $h_i$ ($i \in \mathbf{N}$) as name tags for loops in $A$, and $x_i$ ($i \in \mathbf{N}$) as name tags for unknown loops on $M_0$. That is, $h_i$’s are constants and $x_i$’s are variables. The goal is to describe $x_i$'s in terms of $h_j$’s.
\begin{remark}
The length of $h_i$'s are $6$. The length of $x_i$'s are unknown.
\end{remark}
For example, in Figure \ref{figure2} (a) middle, we have two equations:
\begin{equation*}
\left\{
\begin{aligned}
|x_1x_2x_3x_4|&= |x_3x_5|,\\
 |x_1x_2x_3x_4|&= |x_2x_6|.
\end{aligned}
\right.
\end{equation*} 
Subtracting the common factors yields
\begin{equation*} 
\left\{
\begin{aligned}
|x_1x_2x_4|&= |x_5|,\\
 |x_1x_3x_4|&=|x_6|.
\end{aligned}
\right.
\end{equation*} 
These are the defining system of equations for $x_i$`s. Because $x_1, x_2, x_3$, and $x_4$ are not decomposable (i.e., there is no equation that decomposes $x_1, x_2, x_3$, and $x_4$ into smaller parts). $\exists h_i \in A$  ($i=a, b, c, d$) such that 
\begin{equation*} 
x_1=h_a,\ x_2=h_b,\ x_3=h_c, \text{ and } x_4=h_d.
\end{equation*} 
Now we have
\begin{equation*} \label{eq_int}
\left\{
\begin{aligned}
x_5 &= h_a+h_b+h_d, \\
x_6 &= h_a+h_c+h_d.
\end{aligned}
\right.
\end{equation*} 
Set 
\begin{equation*} 
h_a=h_2,\ h_b=h_3,\ h_c=h_4, \text{ and } h_d=h_5 
\end{equation*} 
(Figure \ref{figure2} (a) left). Then, $\exists \psi_1,\ \psi_2 \in LOOP(M_0)$ such that
\begin{equation*}
\left\{
\begin{aligned}
\psi_1 &= h_2+h_3+h_5, \\
\psi_2 &= h_2+h_4+h_5.
\end{aligned}
\right.
\end{equation*} 
That is, 
\begin{equation*} 
x_5=\psi_1 \text{ and } x_6=\psi_2.
\end{equation*} 
The geometric arrangement of the $h_i$’s are shown in Figure \ref{figure2} (a) left. The arrangement of $h_i$’s may not be unique, but the overall shape is unique modulo rigid motions, i.e., translations, rotations, and reflections. 

\begin{remark}
Hexagons must be placed at specific positions for them to interact. For example, three hexagons interact only when they are placed in a triangular position.
\end{remark}

First, we define ``intergral'' molecules as follows.

\begin{definition} [Integral molecules $INT(Mol)$]
Given $m \in Obj(Mol)$. $m$ is called an \textit{integral molecule} if 
\begin{equation*} 
A \rightarrow_M m. 
\end{equation*} 
$INT(Mol)$ denotes the collection of all integral molecules, i.e.,
\begin{equation*} 
INT(Mol):=\{m \in Obj(Mol) \ |\  A \rightarrow_M m\}.
\end{equation*} 
The collection INT(Loop) of all \textit{integral loops} and the collection INT(Reg) of all \textit{integral regions} are defined by
\begin{align*}
INT(Loop):&=\{ l \in Obj(Loop)\  |\  A \rightarrow_L l\},\\
INT(Reg):&=\{ |m| \in Obj(Reg)\ |\  m \in INT(Mol) \}.
\end{align*}
\end{definition}

\begin{example}
All objects in Figure \ref{figure2} except $x_7$ are integral molecules.
\end{example}

\begin{definition} [$Supp(m)$]
Given $m \in INT(Mol)$. By definition, $\exists \{h_1,\ h_2,\  \ldots,\ h_k\} \subset A$ uniquely such that
\begin{equation*}
m=h_1+ h_2+  \ldots+ h_k.
\end{equation*}
$Supp(m)$ is then defined by
\begin{equation*}
Supp(m):=\{h_1,\ h_2,\  \ldots,\ h_k\} \quad \subset A
\end{equation*}
and called the \textit{support} of $m$.
\end{definition}

\begin{definition} [$INT(U)$]
Given $U \subset A$. $INT(U)$ denotes the collection of all integral molecules whose supports are contained in $U$, i.e.,
\begin{equation*}
INT(U):=\{\phi \in INT(Mol) \cap MOL(M_0)\ |\ Supp(\phi) \subset U \}.
\end{equation*}
\end{definition}

\begin{remark}
Given $m \in INT(Mol)$. Then, $m \in INT(Supp(m))$.
\end{remark}

\begin{figure}
\centering
\captionsetup{width=1.0\linewidth}
\includegraphics{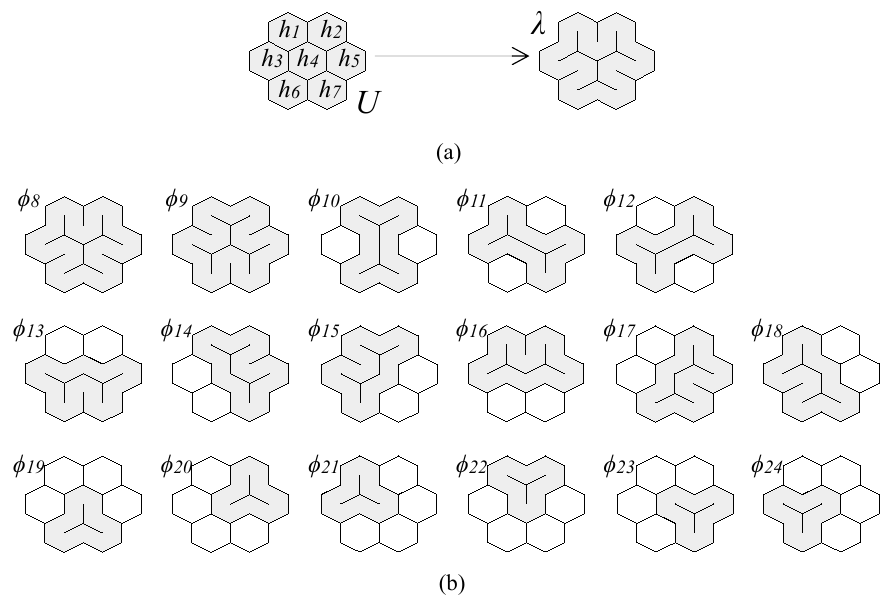}
\caption{Defining system of equations for an integral molecule: (a) $U$ and $\lambda$. (b) $C_U$.}
\label{figure2a}
\end{figure}

\begin{remark} Given $k,l \in \mathbf{Z}$. $[k,l]$ denotes all integers from $k$ to $l$, as well as $k$ and $l$ themselves, i.e.,
\begin{equation*}
[k, l]:=\{k,\ k+1,\ k+2,\ \ldots,\ l-1,\ l \} \quad \subset \mathbf{Z}.
\end{equation*}
\end{remark}

\begin{definition} [Defining system of equations for an integral molecule]\label{def_eq_for_int}
Given $\lambda \in INT(Mol) \cap MOL(M_0)$. Suppose
\begin{align*}
Supp(\lambda)&=\{h_1,\ h_2,\ \ldots, \ h_k\}, \\
INT(Supp(\lambda))&=\{\phi_{k+1},\ \phi_{k+2},\ \ldots,\ \phi_{k+n}\} \quad (\phi_{k+1}=\lambda),\\
Supp(\phi_{j})&=\{h_{j(1)},\ h_{j(2)},\ \ldots, \ h_{j(k_j)}\} \quad \subset Supp(\lambda)
\end{align*}
($j \in [k+1, k+n]$). Then,
\begin{align*}
\phi_{k+1}&=\lambda=h_1+ h_2+  \cdots+ h_k, \\ 
\phi_j&= h_{j(1)}+ h_{j(2)}+ \cdots + h_{j(k_j)}  \quad (j \in [k+2, k+n]).
\end{align*}
By replacing $h_i$ and $\phi_j$ with indeterminates $x_i$ and $y_j$, respectively, we obtain a system of equations:
\begin{equation}\label{eq_defeq}
y_j= x_{j(1)}+ x_{j(2)}+ \cdots + x_{j(k_j)}  \quad (j \in [k+1, k+n]).
\end{equation}
Equation (\ref{eq_defeq}) is called the \textit{defining system of equations} for $\lambda$. $\lambda$ is then obtained by
\begin{equation*}
\lambda=x_1+ x_2+  \ldots+ x_k.
\end{equation*}
\end{definition}

\begin{remark}
Definition \ref{def_eq_for_int} could be formulated in the ``functor category'' \cite{CAT1} (see the Yoneda lemma).
\end{remark}

\begin{example}
In Figure \ref{figure2a},
\begin{align*}
Supp(\lambda)&=\{h_1,\ h_2,\ \ldots, \ h_7\}, \\
INT(Supp(\lambda))&=\{\phi_{8},\ \phi_{9},\ \ldots,\ \phi_{24}\}.
\end{align*}
where
\begin{equation*}
\left\{
\begin{aligned}
\lambda&=\phi_8=h_1+ h_2+ h_3+h_4+h_5+h_6+ h_7, \\
\phi_9&=h_1+ h_2+ h_3+h_4+h_5+h_6+ h_7, \\
\phi_{10}&=h_1+ h_2+h_4+h_6+ h_7, \\
\phi_{11}&=h_1+ h_3+h_4+h_5+ h_7, \\
&\cdots\\
\phi_{23}&= h_4+h_5+h_7,\\
\phi_{24}&= h_3+h_4+h_6.
\end{aligned}
\right.
\end{equation*}
By replacing $h_i$ and $\phi_j$ with indeterminates $x_i$ and $y_j$, respectively, we obtain the 
defining system of equations for $\lambda$.
\end{example}

\begin{remark}
Every triplet of hexagons in a triangular position contributes to Equation (\ref{eq_defeq}). That is, if
$y_j=x_1+x_2+x_3$ is included in Equation (\ref{eq_defeq}), then $x_1$, $x_2$, and $x_3$ are in a triangular position.
\end{remark}

As of now, I have no proof of the following claim.
\begin{claim} \label{claim_eq)}
Equation (\ref{eq_defeq}) uniquely determines the shape of $\lambda$ modulo rigid motions, i.e., translations, rotations, and reflections.
\end{claim}

\begin{figure}
\centering
\captionsetup{width=1.0\linewidth}
\includegraphics{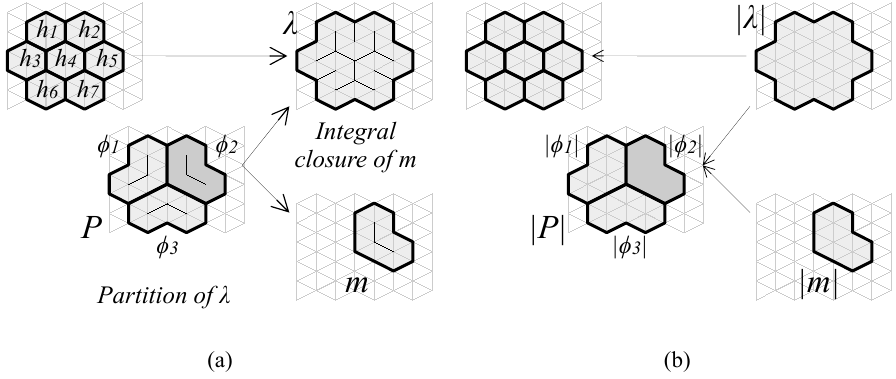}
\caption{(a) Definition of rational molecules. (b) Definition of rational regions.}
\label{figure3}
\end{figure}

\subsubsection{The defining system of equations for rational molecules }\label{sec313}\ 

Next, let's consider the defining system of equations for molecules obtained by ``fission'' of an integral molecule. We define a collection of ''rational'' molecules as follows.

\begin{definition} [$Part(m)$]
Given $P, m \in Obj(Mol)$. $P$ is called a \textit{partition} of $m$ if
\begin{enumerate}
\item $P \rightarrow_M m$, and
\item $|P|=|m|$. 
\end{enumerate} 
Given $m \in Obj(Mol)$. $Part(m)$ denotes the collection of all partitions of $m$, i.e,
\begin{equation*} 
Part(m):=\{P \in Obj(Mol)\ |\  P \text{ is a partition of } m\} \quad \subset Obj(Mol).
\end{equation*}
\end{definition}

\begin{remark}
A mapping $F$ from $REG(M_0)$ to $Power(Obj(Mol))$ is defined by
\begin{equation*} 
F(|m|):=Part(m) \quad \subset Obj(Mol),
\end{equation*}
where $m \in Mol(M_0)$.
\end{remark}

\begin{remark}
$Part(m)$ defines a ``topology'' on $Obj(Mol)$, i.e.,
\begin{equation*} 
|m'| \subset |m|  \text{ if } \exists P \in Part(m) \text{ such that } P \rightarrow_M m'
\end{equation*}
for $m' \in Obj(Mol)$. The opposite is not true.
\end{remark}

\begin{example}
In Figure \ref{figure3} (a), $P \in Part(\lambda)$. The vertical flip of $P$ gives another partition of $\lambda$.
\end{example}

\begin{definition} [$Part(\lambda, m)$]
Given $\lambda, m, \text{and }P \in Obj(Mol)$. $P$ is called a \textit{partition} of $\lambda$ consistent with $m$ if
\begin{enumerate}
\item $P \in Part(\lambda)$, and
\item $P \rightarrow_M m$.
\end{enumerate} 
Given $\lambda, m \in Obj(Mol)$. $Part(\lambda,m)$ denotes the collection of all partitions of $\lambda$ consistent with $m$, i.e,
\begin{equation*} 
Part(\lambda, m):=\{P \in Part(\lambda)\ |\  P \rightarrow_M m \} \quad  \subset Obj(Mol).
\end{equation*}
Note that 
\begin{equation*} 
|m| \subset |\lambda| \quad \text{if } Part(\lambda, m)\neq \emptyset
\end{equation*}
\end{definition}

\begin{definition} [Rational molecules $RAT(Mol)$]
Given $m \in Obj(Mol)$. $m$ is called a \textit{rational molecule} if 
\begin{equation*}
\exists \lambda \in INT(Mol) \cap MOL(M_0) \text{ s.t. } Part(\lambda,m) \neq \emptyset
\end{equation*} 
(Figure \ref{figure3} (a)). $\lambda$ is called an \textit{integral closure} of $m$. $Cl(m)$ denotes the collection of all integral closure of $m$, i.e,
\begin{equation*} 
Cl(m):=\{\lambda \in INT(Mol) \cap MOL(M_0)\ |\ Part(\lambda,m) \neq \emptyset \}.
\end{equation*} 
$RAT(Mol)$ denotes the collection of all rational molecules, i.e., 
\begin{equation*} 
RAT(Mol):=\{m \in Obj(Mol) \ |\ Cl(m) \neq \emptyset\}.
\end{equation*}
The collection $RAT(Loop)$ of rational loops and the collection $RAT(Reg)$ of rational regions are also defined in the same way (Figure \ref{figure3} (b)).
\end{definition}

\begin{remark}[Computation of $Cl(m)$, $Part(\lambda)$, and $Part(\lambda,m)$]
We can compute elements of $Cl(m)$, $Part(\lambda)$, and $Part(\lambda,m)$ quickly by ``cube stacking'' (see Subsection \ref{sec325}).
\end{remark}

\begin{lemma}\label{INT_RAT_Obj}
$INT(Mol) \subset RAT(Mol) \subset Obj(Mol)$.
\end{lemma}
\begin{proof}
It follows immedeately from the definitions.
\end{proof}

As of now, I have no proof of the following claim.

\begin{claim}\label{RAT_Obj}
$RAT(Mol) = Obj(Mol)$.
\end{claim}

\begin{figure}
\centering
\captionsetup{width=1.0\linewidth}
\includegraphics{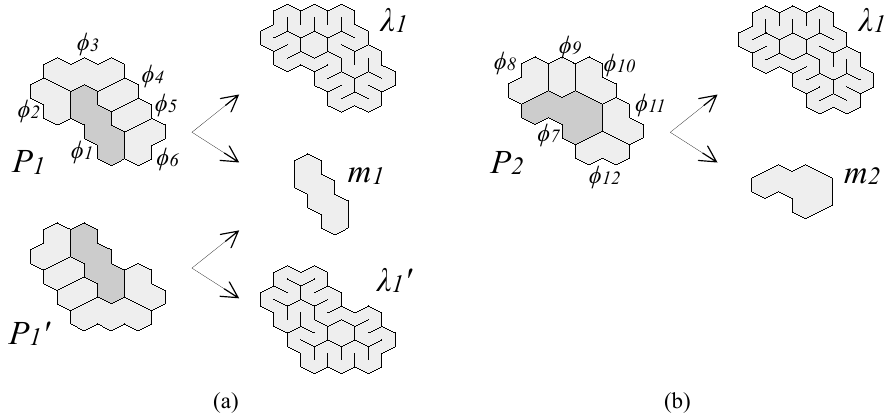}
\caption{Examples of monochromatic spectra.}
\label{figure3a}
\end{figure}

\begin{definition} [Minimal integral closures]
Given $m \in RAT(Mol)$ and $\lambda \in Cl(m)$. $\lambda$ is called \textit{minimal} if
\begin{equation*}
|\lambda’| \subset |\lambda| \text{ implies } \lambda’=\lambda
\end{equation*}
for $\forall \lambda’ \in Cl(m)$.
\end{definition}

\begin{definition} [$Spec(m)$]
Given $m \in RAT(Mol)$. $Spec(m)$ denotes the collection of all minimal integral closures of $m$, i.e.,
\begin{equation*}
Spec(m):=\{\lambda \in Cl(m) \ |\ \lambda \text{ is minimal} \} \quad \subset INT(Mol) \cap MOL(M_0)
\end{equation*}
Elements of $Spec(m)$ are called the \textit{spectra} of $m$. 
\end{definition}

\begin{example} 
In Figure \ref{figure3} (a), 
\begin{enumerate}
\item $\lambda = h_1+h_2+ \cdots +h_7 \in Spec(m)$,
\item $P=\phi_1\phi_2\phi_3 \not\in Cl(m)$ because $P$ is neither integral nor a loop.
\end{enumerate}
\end{example} 

\begin{lemma}
Given $m \in RAT(Mol)$ such that $Spec(m)=\{\lambda_1,\ \lambda_2,\ \ldots,\ \lambda_k \}$ ($k \in \mathbf{N}$). Then,
\begin{equation*}
\text{for }\forall \kappa \in Cl(m),\  \exists \lambda_i \in Spec(m) \text{ and } b \in INT(Mol) \text{ such that }  \kappa = \lambda_i+b.
\end{equation*}
\end{lemma}
\begin{proof}
It follows immediately from the definitions.
\end{proof}

We identify the spectra that can be transformed into each other by ``rigid motions.''

\begin{definition} [$RM(Obj(Mol))$]\label{def_RM}
A \textit{rigid motion} is a geometric operation on a solid object that changes its position and orientation in space while maintaining its shape, size, but not its chirality ("handedness"). Rigid motions include translations and rotations, and reflections. $RM(Obj(Mol))$ denotes the collection of all rigid motions of elements of $Obje(Mol)$, i.e.,
\begin{multline*}
RM(Obj(Mol)):=\{ \tau \in MAP(Obj(Mol)) \ |\  \tau \text{ is a rigid motion} \}\\
 \subset MAP(Obj(Mol)),
\end{multline*}
where $MAP(Obj(Mol))$ denotes the collection of all mappings from $Obj(Mol)$ to $Obj(Mol)$. 
Given $m \in Obj(Mol)$ and $\tau \in RM(Obj(Mol))$. $\tau(m)$ denotes the molecule obtained by transforming $m$ by $\tau$. 
\end{definition}

\begin{definition} [Monochromatic spectrum]
Given $m \in RAT(Mol)$ such that $Spec(m)=\{\lambda_1,\ \lambda_2,\ \ldots,\ \lambda_k \}$ ($k \in \mathbf{N}$). $Spec(m)$ is called \textit{monochromatic} if elements of $Spec(m)$ can be transformed into each other by rigid motions, i.e.,
\begin{equation*}
\text{for }\forall \lambda_i, \lambda_j \in Spec(m),\  \exists \tau \in RM(Obj(Mol)) \text{ such that }  \lambda_j=\tau(\lambda_i).
\end{equation*}
We write 
\begin{equation*}
Spec(m)=\{\lambda_i\} \bmod \ RM(Obj(Mol))
\end{equation*}
if $Spec(m)$ is monochromatic ($\lambda_i$ is any element of $Spec(m)$).
\end{definition}

\begin{example}
In Figure \ref{figure3a} (a), 
\begin{equation*}
Spec(m_1)=\{ \lambda_1, {\lambda_1}' \}.
\end{equation*}
They are transformed into each other by ``$180$ degree rotation.'' Therefore, $Spec(m_1)$ is monochromatic, i.e., 
\begin{equation*}
Spec(m_1)=\{ \lambda_1 \} \bmod \ RM(Obj(Mol)),
\end{equation*}
\end{example}

\begin{remark}
In Figure \ref{figure3a}, $Spec(m_2)=\{\lambda_1\}$. That is, different rational molecules may have the same spectrum.
\end{remark}

\begin{remark} 
Given $m, c \in Obj(Mol)$. $c$ is called the \textit{component} of $m$ if
\begin{equation*}
m \rightarrow_M c.
\end{equation*}
Given $m \in Obj(Mol)$. $Comp(m)$ denotes the collection of all components of $m$, i.e.,
\begin{equation*} 
Comp(m):=\{c \in Obj(Mol)\ |\  m \rightarrow_M c\} \quad \subset Obj(Mol).
\end{equation*}
Given $m_1, m_2 \in Obj(Mol)$. Then, $m_1$ and $m_2$ can interact if
\begin{equation*}
\exists \lambda \in MOL(M_0),\ P \in Part(\lambda) \text{ s.t. } m_1, m_2 \in Comp(P).
\end{equation*}
Note that the interaction of $m_1$ and $m_2$ may require the addition of other molecules (i.e., ``allosteric regulation'' \cite{NM1}).
\end{remark} 

\begin{remark} 
Given $m$, $P \in Obj(Mol)$ such that $P \in Part(m)$. Roughly speaking, a molecule $m$ corresponds to an ``inductive limit,'' and a partition $P$ corresponds to a ``projective limit,'' i.e.,
\begin{align*}
m  &\leftrightarrow \varinjlim_{Q \in Part(m)}Q,\\
P &\leftrightarrow \varprojlim_{\iota \in Comp(P)} \iota.
\end{align*}
In particular,
\begin{equation*}
A \leftrightarrow \varprojlim_{m \in INT(Mol)} m.
\end{equation*}
\end{remark}

\begin{remark} 
In \cite{NM2}, \textit{$L$-spectrum} of $U \in RAT(Mol)$ is defined by
\begin{multline*}
L\text{-}spec(U):=\{ V \in INT(Mol)\  |\ 
\exists W \in RAT(Mol) \text{ such that }\\
 W \rightarrow_M V \text{ and } W \rightarrow_M U \}.
\end{multline*} 
\textit{Topology} $K$ on $RAT(Mol)$ is defined by
\begin{equation*}
K(U):=\{W \in RAT(Mol)\  |\  W \rightarrow_M U \} .
\end{equation*}
A collection $F(U)$ of loops on $U$ is defined by
\begin{equation*}
F(U):= \{W \in RAT(Mol)\  |\  U \rightarrow_M W \}.
\end{equation*}
\end{remark}

In summary, if we measure the shape of $m \in RAT(Mol)$ using the hexagons of $A$ (with respect to $\rightarrow_M$), we obtain an integral molecule $\lambda \in Spec(m)$. $m$ is then obtained by dividing $\lambda$ into parts. 

\begin{remark}[Quantum observation of $m$]
Given $m \in Obj(Mol)$ and $\lambda \in Spec(m)$. In our framework, information about $m$ is obtained through $\lambda$. That is, $\lambda$ is ``observed'' as a result of the ``measurement'' of $m$.
\end{remark}

Now, a defining system of equations for a rational molecule $m \in RAT(Mol)$ is obtained using its spectra $Spec(m) \subset INT(Mol) \cap MOL(M_0)$.

Given $m \in RAT(Mol)$. Let
\begin{align*}
Spec(m)&=\{ \lambda_1,\  \lambda_2,\  \dots,\ \lambda_n\}  \quad (n \in \mathbf{N})\\
Part(\lambda_i, m)&=\{ P_{i, 1},\  P_{i,2},\  \ldots,\  P_{i,n_i}\}  \quad (n_i \in \mathbf{N},\ i=1,2,\ldots,n),\\
P_{i,j}&={\phi}_{i,j,1}{\phi}_{i,j,2} \cdots {\phi}_{i,j,k_{i,j}} \quad (k_{i,j} \in \mathbf{N},\ j \in [1, n_i],\  i \in [1,n]).
\end{align*}
For each pair of $\lambda_i \in Spec(m)$ and $P_{i,j} \in Part(\lambda_i,m)$, a system of equations is obtained as follows.

Pick a pair $\lambda_i$ and $P_{i,j}$. By definition, $\exists l_{i,j} \in [1,k_{i,j}]$ such that
\begin{equation*} 
\left\{
\begin{aligned}
m&={\phi}_{i,j,l_{i,j}}, \\
\lambda_i&= m + \sum_{k \in [1,k_{i,j}] \setminus \{l_{i,j}\}} {\phi}_{i,j,k}.
\end{aligned}
\right.
\end{equation*} 
In particular,
\begin{equation*}
|\lambda_i|=||m| |\prod_{k \in [1,k_{i,j}] \setminus \{l_{i,j}\}} {\phi}_{i,j,k}||.
\end{equation*}
By abuse of notation, we write
\begin{equation*}
|m|=|\lambda_i|/ |\prod_{k \in [1,k_{i,j}] \setminus \{l_{i,j}\}} {\phi}_{i,j,k}|
\end{equation*}
dividing both sides by $|\prod_{k \in [1,k_{i,j}] \setminus \{l_{i,j}\}} {\phi}_{i,j,k}|$, where
\begin{align*}
Spec(m)&=Spec(\phi_{i,j,l_{i,j}})=\{ \lambda_1,\  \lambda_2,\  \dots,\ \lambda_n\},\\
Spec(\phi_{i,j,k})&=\{{\lambda}_{i,j,k,1},\ {\lambda}_{i,j,k,2},\ \ldots ,\ {\lambda}_{i,j,k,g_{i,j,k}} \}
\end{align*}
($g_{i,j,k} \in \mathbf{N} ,\ k \in [1,k_{i,j}] \setminus \{ l_{i,j}\}$).

\begin{remark}
In the normal notation, this is the set difference, i.e., $|m|=|\lambda| \setminus |\prod_{k \neq l} {\phi}_{i,j,k}|$. The multiplicative notation is used here because region complices are denoted by juxtaposition (see the last paragraph of Section \ref{sec2}).
\end{remark}

Replacing $m$ and $ {\phi}_{i,j,k}$’s with indeterminants $x$ and $x_k$’s, respectively, we obtain a collection of equations:
\begin{equation}\label{def_eq2}
\left\{
\begin{aligned}
|x|&=|\lambda_i|/ |\prod_{k \in [1,k_{i,j}] \setminus\{l_{i,j}\}} x_{k}|, \\
Spec(x)&=\{{\lambda}_{1},\ {\lambda}_{2},\ \ldots ,\ {\lambda}_{n} \},\\
Spec(x_k)&=\{{\lambda}_{i,j,k,1},\ {\lambda}_{i,j,k,2},\ \ldots ,\ {\lambda}_{i,j,k,g_{i,j,k}} \}
\end{aligned}
\right.
\end{equation}
($k \in [1,k_{i,j}] \setminus \{ l_{i,j}\},\ j \in [1, n_i],\  i \in [1,n]$).
\begin{remark}
$Spec(x_k)$ gives the values observed when measuring $x_k$.
\end{remark}

\begin{figure}
\centering
\captionsetup{width=1.0\linewidth}
\includegraphics{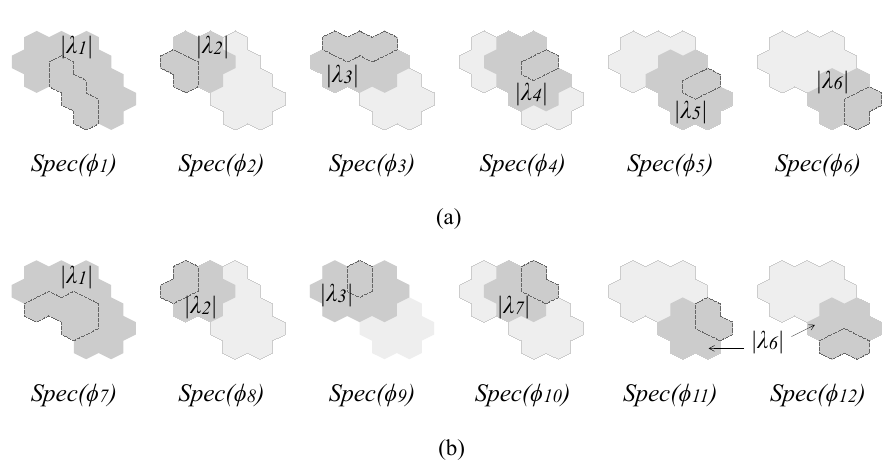}
\caption{Spectra of rational molecules of Figure \ref{figure3a}.}
\label{figure3b}
\end{figure}

\begin{definition} [The defining system of equations for a rational molecule (1)]\label{def_shape1}
\textbf{The case when $Spec(x_{k})=Spec(x_{k'})$ for $\forall k,\ k' \in [1,k_i]$:} Equation(\ref{def_eq2}) yields
\begin{equation*}
\left\{
\begin{aligned}
|x|&=|\lambda_i|/ |\prod_{k \in [1,k_{i,j}] \setminus\{l_{i,j}\}} x_{k}|, \\
Spec(x)&=Spec(x_{k})=\{{\lambda}_{1},\ {\lambda}_{2},\ \ldots ,\ {\lambda}_{n} \}
\end{aligned}
\right.
\end{equation*}
($k \in [1,k_{i,j}] \setminus \{ l_{i,j}\},\ j \in [1, n_i],\  i \in [1,n]$). Since $x_k$'s are indistinguishable by their spactra, we identify them all. The defining system of equations for $m$ is then given by
\begin{equation*}
|x^{k_{i,j}}|=|\lambda_i|
\end{equation*}
($j \in [1, n_i],\  i \in [1,n]$).
\end{definition}

\begin{example}
In Figure \ref{figure3} (a), $Spec(m)=Spec(\phi1)=Spec(\phi3)=\{\lambda\}$. As the defining equation for $C$, we obtain
\begin{equation*}
|x^3| =|\lambda|,
\end{equation*}
where $\lambda=h_1+h_2+ \cdots +h_7 \in INT(Mol) \cap MOL(M_0)$.
\end{example}

\begin{definition} [The defining system of equations for a rational molecule (2)]\label{def_shape2}
\textbf{The case when all $Spec(x_i)$'s are monochromatic:} Equation(\ref{def_eq2}) yields
\begin{equation*}
\left\{
\begin{aligned}
|x|&=|\lambda_i|/ |\prod_{k \in [1,k_{i,j}] \setminus\{l_{i,j}\}} x_{k}|, \\
Spec(x)&=\{{\lambda}_{1},\ {\lambda}_{2},\ \ldots ,\ {\lambda}_{n} \}, \\
Spec(x_k)&=\{{\lambda}_{i,j,k}\}  \bmod \ RM(Obj(Mol))
\end{aligned}
\right.
\end{equation*}
($k \in [1,k_{i,j}] \setminus \{ l_{i,j}\},\ j \in [1, n_i],\  i \in [1,n]$). By identifying $x_i$ with $\lambda_{i,j,k}$, we obtain the defining system of equations for $m$ by
\begin{equation*}
|x|=|\lambda_l|/ |\prod_{k \in [1,k_{i,j}] \setminus\{l_{i,j}\}} \lambda_{i,j,k}|
\end{equation*}
($j \in [1, n_i],\  i \in [1,n]$).
\end{definition}

\begin{example}
In the case of Figure \ref{figure3a} (a), $Spec(x)$ and all $Spec(x_i)$'s are monochromatic. As the defining equation for $m_1$, we obtain
\begin{equation*}
|x|=|\lambda_1|/ |\lambda_2\lambda_3\lambda_4\lambda_5\lambda_6|,
\end{equation*}
where $\lambda_i$'s are given in Figure \ref{figure3b} (a).
\end{example}

\begin{example}
In the case of Figure \ref{figure3a} (b), $Spec(x)$ all $Spec(x_i)$'s are monochromatic. As the defining equation for $m_2$, we obtain
\begin{equation*}
|x|=|\lambda_1|/ |\lambda_2\lambda_3\lambda_7\lambda_6^2|,
\end{equation*}
where $\lambda_i$'s are given in Figure \ref{figure3b} (b).
\end{example}

\begin{definition} [The defining system of equations for a rational molecule (3)]\label{def_shape3}
\textbf{The case when $Spec(m)$ is monochromatic:} Equation(\ref{def_eq2}) yields
\begin{equation*}
\left\{
\begin{aligned}
|x|&=|\lambda|/ |\prod_{k \in [1,k_{j}] \setminus\{l_{j}\}} x_{k}|, \\
Spec(x)&=\{{\lambda}\}  \bmod \ RM(Obj(Mol)), \\
Spec(x_k)&=\{{\lambda}_{j,k,1},\ {\lambda}_{j,k,2},\ \ldots ,\ {\lambda}_{j,k,g_{j,k}} \}.
\end{aligned}
\right.
\end{equation*}
($k \in [1,k_{j}] \setminus \{ l_{j}\},\ j \in [1, n]$). For each combination 
\begin{equation*}
(\lambda_1,\ \ldots, \lambda_{l_j-1},\ \lambda_{l_j+1},\ \ldots, \lambda_{k_j}) \in 
\prod_{k \in [1,k_{j}] \setminus\{l_{j}\}} Spec(x_k)
\end{equation*}
($j \in [1, n]$), we obtain an equation for $m$.
\end{definition}

\begin{definition} [The defining system of equations for a rational molecule (4)]\label{def_shape4}
\textbf{Other cases:} For each $\lambda_i \in Spec(m)$, each combination 
\begin{equation*}
(\lambda_1,\ \ldots, \lambda_{l_{i,j}-1},\ \lambda_{l_{i,j}+1},\ \ldots, \lambda_{k_{i,j}} \in 
\prod_{k \in [1,k_{i,j}] \setminus\{l_{i,j}\}} Spec(x_k)
\end{equation*}
($j \in [1, n_i],\  i \in [1,n]$) yields an equation for $m$.
\end{definition}

Molecules are now described using integral molecules, just as rational numbers are written using integers. The question is, as mentioned at the beginning of this subsection, ``To what extent can a molecule be specified by integral molecules?''
As of now, I have no answer to the question.

\begin{figure}
\centering
\captionsetup{width=1.0\linewidth}
\includegraphics{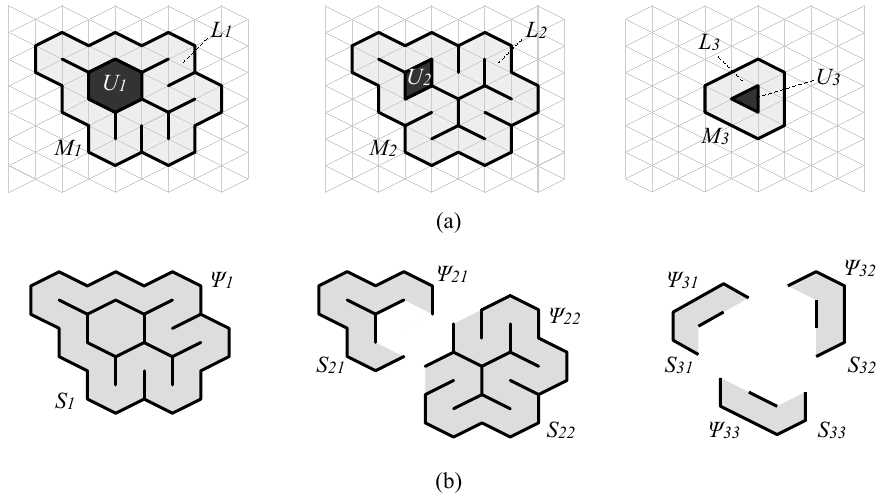}
\caption{(a) Three types of holes within a loop. (b) The parts of molecules to patch.}
\label{figure4}
\end{figure}

\subsection{Is the given loop a molecule?}\label{sec32}\ 

Using concepts from cohomology theory, we derive an equation that determines whether holes within a loop can be decomposed into loops.

\subsubsection{Singularities of a loop }\label{sec321}\ 

Given $M \subset M_0$ (Definition \ref{def_mesh}) and $L \in LOOP(M)$ (Definition \ref{def_loop}). Recall that loops may have holes inside, i.e., ``singularities.'' 

\begin{definition} [Singularities]
Given $M \subset M_0$, $L \in LOOP(M)$, and $U \in HOLE(L)$ (Definition \ref{def_hole}). $U$ is called \textit{singular} if it cannot be decomposed into loops. A singular region in $HOLE(L)$ is called a \textit{singularity} of $L$.
\end{definition}

Given $M \subset M_0$, $\psi_1, \psi_2 \in MOL(M)$ (Definition \ref{def_mol}), and $\psi_3 \in LOOP(M)$. Suppose $|\psi_3|=|\psi_1| \cup |\psi_2|$. Then, by Lemma \ref{INT_and_LD}, $\psi_1$ and $\psi_2$ interact if all regions in $HOLE(\psi_3)$ are not singular. 
Therefore, to determine whether $\psi_1$ and $\psi_2$ interact, it is necessary to determine whether all holes within $\psi_3$ are singular or not. 

\begin{example}\label{ex_holes}
Figure \ref{figure4} (a) shows three types of holes within a loop. On the left, 
\begin{equation*}
HOLE(L_1)=\{U_1\}.
\end{equation*}
Since $U_1$ can be decomposed into a loop of length $6$ (Figure \ref{figure4} (b) left), $U_1$ is not singular. Therefore, $L_1$ is a molecule.

In the middle, 
\begin{equation*}
HOLE(L_2)=\{U_2\}.
\end{equation*}
Since $U_2$ is singular, $L_2$ is not a molecule.

On the right, 
\begin{equation*}
HOLE(L_3)=\{U_3\}.
\end{equation*}
Since $U_3$ is singular, $L_3$ is not a molecule.

As explained below (Example \ref{ex_19} and Example \ref{ex_20}), $L_2$ is obtained by patching parts of two molecules together, and $L_3$ is obtained by patching parts of three molecules together (Figure \ref{figure4} (b) middle and right).
\end{example}

\subsubsection{Flows of Triangles} \label{sec322}

\paragraph{3.2.2.1 \textit{Regular Flows}}\label{para3221} 
Given $M \subset M_0$ and $R \subset REG(M)$ (Definition \ref{def_region}). We define ``flows'' on $M$ to consider loop decompositions of $R$.

\begin{definition} [$FLOW(M)$]
Given $M \subset M_0$ and $\Psi \subset TRAJ(M)$ (Definition \ref{def_traj}). $\Psi$ is called a \textit{flow} on $M$ if the trajectories of $\Psi$ do not intersect and $\Psi$ covers the entire $M$. i.e.,
\begin{enumerate}
\item For $\forall \psi_1,\psi_2 \in \Psi$, $\psi_1$ and $\psi_2$ share no triangles if  $\psi_1 \neq \psi_2$,
\item For $\forall abc \in Tri(M),\ \exists \psi \in \Psi$ uniquely such that $abc \in \psi$.
\end{enumerate}
A flow may have isolated triangles, terminal triangles, regular triangles, and branch triangles. $FLOW(M)$ denotes the collection of all flows on $M$. 
\end{definition}

\begin{definition}[Regular flows on $M$]
Given $M \subset M_0$ and $\Psi \in FLOW(M)$. $\Psi$ is called \textit{regular on $M$} if all trajectories of $\Psi$ have no isolated, no terminal, or no branch triangles. Note that \underline{$M$ may have holes inside}.
\end{definition}

\paragraph{3.2.2.2 \textit{Collective Flows}}\label{para3223}
Given $M \subset M_0$. By patching flows on submeshes of $M$ together, we obtain a flow on $M$.

\begin{definition} [$|M|$, $Bd(M)$, and $Int(M)$]
Given $M \subset M_0$. $|M|$ denotes the region occupied by $M$, i.e.,
\begin{equation*}
|M|:=\bigcup_{abc \in Tri(M)} abc \in REG(M).
\end{equation*}
$Bd(M)$ denotes the \textit{boundary} of $|M|$ i.e.,
\begin{equation*}
Bd(M):=\bigcup_{ab \in E(M)} ab \subset \mathbf{R}^2,
\end{equation*}
where
\begin{multline*}
E(M):=\{ab \in Edge(M) \ | \exists c \in Tri(M_0) \text{ such that }\\
 abc \in Tri(M_0) \setminus Tri(M)\}.
\end{multline*}
$Int(M)$ denotes the \textit{interior} of $|M|$, i.e., 
\begin{equation*}
Int(M):=|M| \setminus Bd(M) \subset \mathbf{R}^2.
\end{equation*}
\end{definition}

\begin{definition}[Open covering of $M$]\label{def_covering}
Given $M \subset M_0$ and $S_i \subset M$ ($i \in I$), where $I \subset \mathbf{Z}$. Set $V=\{S_i \ |\ i \in I \}$. 
$V$ is called an \textit{open covering} of $M$ if 
\begin{enumerate}
\item $Int(M)=\bigcup_{i \in I} Int(S_i)$, and
\item $|S_i|$ has no hole inside ($i \in I$). 
\end{enumerate}
\end{definition}

\begin{remark}
In the following, the index set $I$ of an open covering is always a subset of $\mathbf{Z}$, i.e., $I \subset \mathbf{Z}$.
\end{remark}

\begin{definition}[$\Psi|_S$]
Given $S, M \subset M_0$ such that $S \subset M$. Given $\Phi \in FLOW(S)$ and $\Psi \in FLOW(M)$, where
\begin{align*}
\Phi&=\{\phi_{i} \in TRAJ(S) \ |\ i\in I_S\} \quad (I_S \subset \mathbf{Z}),\\
\Psi&=\{\psi_{j} \in TRAJ(M) \ |\ j \in I_M\} \quad (I_M \subset \mathbf{Z}). 
\end{align*}
$\Phi$ is called the \textit{restriction of $\Psi$ on $S$} if
\begin{equation*}
\text{for }\forall \phi_{i} \in \Phi,\  \exists \psi_{j} \in \Psi \text{ such that }  \phi_{i} \subset \psi_{j}. 
\end{equation*}
Recall that we write $\phi_{i} \subset \psi_{j}$ if $abc \in \psi_{j}$ for $\forall abc \in \phi_{i}$ (Definition \ref{def_traj}).
The restriction of $\Psi$ on $S$ is uniquely determined and denoted by ${\Psi}|_{S}$, i.e.,
\begin{equation*}
 \Phi={\Psi}|_{S} \quad  \text{if $\Phi$ is the restriction of $\Psi$ on $S$}.
 \end{equation*}
\end{definition}

\begin{example}\label{ex_17}
In Figure \ref{figure4} (a) left, $L_1$ is a loop of length $42$ on $M_0$ such that $HOLE=\{U_1\}$. 
Define $M_1,\ M_{U_1} \subset M_0$ by
\begin{align*}
Tri(M_1)&=\{abc \in Tri(M_0)\  |\  abc \in L_1\},\\
Tri(M_{U_1})&=\{abc \in Tri(M_0) \  |\  |abc| \subset U_1\}.
\end{align*}
Then, $\Psi_1 \in FLOW(M_1\cup M_{U_1})$ (Figure \ref{figure4} (b) left) and
\begin{equation*}
\Psi_1|_{M_1}=\{L_1\} \in FLOW(M_1).
\end{equation*}
\end{example}

\begin{definition}[Consistent collection of flows] \label{def_consistancy21}
Given $M \subset M_0$, its open covering $V=\{S_i \subset M\  |\  i \in I \}$, and locally defined flows 
\begin{equation*}
CF=\{{\Psi}_i \in FLOW(S_i) \ |\ i \in I\}.
\end{equation*}
$CF$ is called \textit{consistent} if
\begin{equation*}
{\Psi_i}|_{S_i \cap S_j}={\Psi_j}|_{S_i \cap S_j} \quad  (i, j \in I).
\end{equation*}
\end{definition}

\begin{lemma}[Collective flows] \label{lem_colletiveflows}
Given $M \subset M_0$, its open covering $V=\{S_i \subset M \  |\  i \in I \}$, and locally defined flows 
\begin{equation*}
CF=\{{\Psi}_i \in FLOW(S_i) \   |\  i \in I \}.
\end{equation*}
Suppose $CF$ is consistent. Then, a flow $\Phi \in FLOW(M)$ is defined by
\begin{equation*}
\Phi|_{S_i}:=\Psi_{i} \in FLOW(S_i) \quad (i \in I).
\end{equation*}
$\Phi$ is called the \textit{collective} flow on $M$ (defined by $CF$) and denoted as a formal sum of locally defined flows , i.e.,
\begin{equation*}
\Phi=\sum_{\Psi_i \in CF} \Psi_i \quad \in FLOW(M).
\end{equation*}
\end{lemma}

\begin{proof}
It follows immediately from the definitions.
\end{proof}

\begin{corollary}\label{cor_colletiveflows}
$\Psi$ is regular on $M$ if and only if $\Psi_i$ is regular on $S_i$ for $\forall i \in I$.
\end{corollary}
\begin{proof}
It follows immediately from the definitions.
\end{proof}

\begin{example}\label{ex_18}
In Figure \ref{figure4} (a) left (Example \ref{ex_17}), set 
\begin{align*}
&M=M_1 \cup M_{U_1},\\ 
&V=\{S_1\}, \\ 
&CF=\{\Psi_1 \}, 
\end{align*}
where $S_1=M_1 \cup M_{U_1}$ and $\Psi_1 \in FLOW(S_1)$ (Figure \ref{figure4} (b) left). Then,
\begin{equation*}
\{L_1\}=\Psi_1|_{S_1}.
\end{equation*}
\end{example}

\begin{example}\label{ex_19}
In Figure \ref{figure4} (a) middle, $L_2$ is a loop of length $46$ on $M_0$ such that $HOLE(L_2)=\{U_2\}$. Define $M_2, M_{U_2} \subset M_0$ by
\begin{align*}
Tri(M_2)&=\{abc \in Tri(M_0)\  |\  abc \in L_2\},\\
Tri(M_{U_2})&=\{abc \in Tri(M_0) \  |\  |abc| \subset U_2\}.
\end{align*}
Then, $\{L_2\} \in FLOW(M_2)$. 
Set 
\begin{align*}
&M=M_2,\\
&V=\{S_{21},\ S_{22}\},\\
&CF=\{\Psi_{21}, \Psi_{22} \},
\end{align*}
where $\Psi_{2i} \in FLOW(S_{2i})$ ($i=1,2$) (Figure \ref{figure4} (b) middle). Then,
\begin{equation*}
\{ L_2\}= \Psi_{21} + \Psi_{22}.
\end{equation*}
\end{example}

\begin{example}\label{ex_20}
In Figure \ref{figure4} (a) right, $L_3$ is a loop of length $12$ on $M_0$ such that $HOLE(L_3)=\{U_3\}$. 
Define $M_3, M_{U_3} \subset M_0$ by
\begin{align*}
Tri(M_3)&=\{abc \in Tri(M_0)\  |\  abc \in L_3\},\\
Tri(M_{U_3})&=\{abc \in Tri(M_0) \  |\  |abc| \subset U_3\}.
\end{align*}
Then, $\{L_3\} \in FLOW(M_3)$. Set 
\begin{align*}
&M=M_3, \\
&V=\{S_{31},\ S_{32},\ S_{33}\},\\
&CF=\{\Psi_{31},\ \Psi_{32},\ \Psi_{33} \},
\end{align*}
where $\Psi_{3i} \in FLOW(S_{3i})$ ($i=1,2,3$) (Figure \ref{figure4} (b) right). Then,
\begin{equation*}
\{L_3\}=\Psi_{31} + \Psi_{32}  + \Psi_{33}.
\end{equation*}
\end{example}

\subsubsection{Potential Functions, Vector Fields, and Fluxes on $M$}\label{sec323}\ 

Given $M \subset M_0$ and $\Psi \in FLOW(M)$. We define $\mathbf{Z}$-valued linear functions on $Ver(M)$, $Edge(M)$, and $Tri(M)$ to infer local structures (e.g., vector fields) of $\Psi$ from global objects (e.g., potential functions) definded on $\Psi$.

\begin{definition} [Potential functions $C^0(M)$]
Given $M \subset M_0$. $C^0(M)$ is the collection of all $\mathbf{Z}$-valued linear functions defined on $C_0(M)$, i.e.,
\begin{equation*}
C^0(M):=Hom(C_0(M), \mathbf{Z}),
\end{equation*}
where $C_0(M)$ is the collection of all formal finite sums of the vertices of $M$ i.e.,
\begin{multline*}
C_0(M):=\{ n_1v_1 + n_2v_2 + \ldots + n_kv_k  \ |\   k \in \mathbf{N},\\ 
n_i  \in \mathbf{Z},\  v_i \in Ver(M) \ (i=1, 2, \ldots, k))\}.
\end{multline*}
Elements of $C^0(M)$ are called \textit{potential functions} (or \textit{$0$-forms}) on $M$. 
\end{definition}

\begin{definition} [Vector fields $C^1(M)$]
Given $M \subset M_0$. $C^1(M)$ is the collection of all $\mathbf{Z}$-valued linear functions defined on $C_1(M)$, i.e.,
\begin{equation*}
C^1(M):=Hom(C_1(M), \mathbf{Z}),
\end{equation*}
where $C_1(M)$ is the collection of all formal finite sums of the \underline{oriented} edges of $M$ i.e.,
\begin{multline*}
C_1(M):=\{ n_1e_1 + n_2e_2 + \ldots + n_ke_k  \ |\   k \in \mathbf{N},\\
n_i  \in \mathbf{Z},\ e_i \in Edge(M) \ (i=1, 2, \ldots, k)\},
\end{multline*}
where $ab \in Edge(M)$ changes its sign when the vertices are swapped, i.e., $ba=-ab$. For example, 
\begin{equation*}
\beta(ba)=-\beta(ab) 
\end{equation*}
for $\beta \in C^1(M)$ and $ab \in Edge(M)$. Elements of $C^1(M)$ are called \textit{vector fields} (or \textit{$1$-forms}) on $M$. 
\end{definition}

\begin{definition} [Fluxes $C^2(M)$]
Given $M \subset M_0$. $C^2(M)$ is the collection of all $\mathbf{Z}$-valued linear functions defined on $C_2(M)$, i.e.,
\begin{equation*}
C^2(M):=Hom(C_2(M), \mathbf{Z}),
\end{equation*}
where $C_2(M)$ is the collection of all formal finite sums of the \underline{oriented} triangles of $M$ i.e.,
\begin{multline*}
C_2(M):=\{ n_1t_ + n_2t_2 + \ldots + n_kt_k  \ |\   k \in \mathbf{N},\\
n_i  \in \mathbf{Z},\  t_i \in Tri(M)\  (i=1, 2, \ldots, k)\},
\end{multline*}
where $abc \in Tri(M)$ changes its sign when the vertices are swapped, i.e., $bac=-abc$, $acb=-abc$, etc. For example, 
\begin{equation*}
\gamma(bac)=-\gamma(abc),\  \gamma(acb)=-\gamma(abc),\  \ldots
\end{equation*}
for $\gamma \in C^2(M)$ and $abc \in Tri(M)$. Elements of $C^2(M)$ are called \textit{fluxes} (or \textit{$2$-forms}) on $M$. 
\end{definition}

\subsubsection{Differential Structure on $M$}\label{sec324}\    

Given $M \subset M_0$ and $\Psi \in FLOW(S)$. We define a differential structure on $M$ to consider local structures of $\Psi$.

\begin{figure}
\centering
\captionsetup{width=1.0\linewidth}
\includegraphics{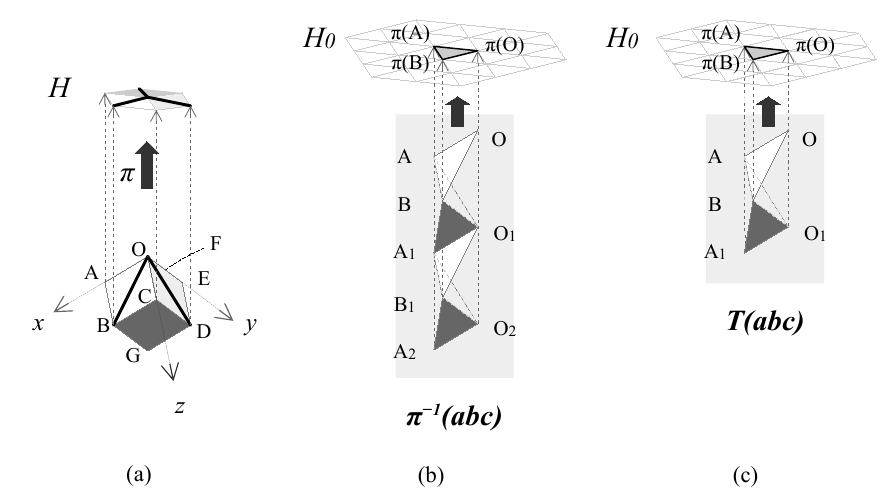}
\caption{The differential structure on $M$.}
\label{figure5}
\end{figure}

\paragraph{3.2.4.1 \textit{Tangent Space $T(M)$}}\label{para3241} 
First, we embed $M_0$ (Definition \ref{def_mesh}) into $\mathbf{R}^3$.

\begin{definition}[$Tri(\mathbf{R}^3)$]
$Tri(\mathbf{R}^3)$ is the collection of triangles in $\mathbf{R}^3$ defined by
\begin{equation*}
Tri(\mathbf{R}^3):= \{ OAB,\ OBC,\ OCD,\ ODE,\ OEF, \text{ and } OFA \ |\  (l, m, n)\in \mathbf{Z}^3 \},
\end{equation*}
where
\begin{align*}
&O=(l,\ m,\ n),\ A=(l+1,\ m,\ n),\ B=(l+1,\ m+1,\ n),\\
&C=(l,\ m+1,\ n),\ D=(l,\ m+1,\ n+1),\ E=(l,\ m,\ n+1),\\
&F=(l+1,\ m,\ n+1), \text{ and }G=(l+1,\ m+1,\ n+1) \in \mathbf{Z}^3
\end{align*}
(Figure \ref{figure5} (a)). The elements of $Tri(\mathbf{R}^3)$ are called \textit{slant} triangles.
\end{definition}

\begin{definition}[Mesh $H_0$]\label{def_H0}
Let $H \subset \mathbf{R}^3$ be the hyperplane defined by
\begin{equation*}
H:=\{(x,\ y,\ z) \in \mathbf{R}^3 \ |\  x+y+z=0\}.
\end{equation*}
Let $\pi:\  \mathbf{R}^3 \rightarrow H$ be the vertical projection defined by
\begin{equation*}
\pi(x,\ y,\ x):=((2x-y-z)/3,\ (-x+2y-z)/3,\ (-x-y+2z)/3).
\end{equation*}
Then, $\pi$ induces a projection of $Tri(\mathbf{R}^3)$ on $H$, also denoted by $\pi$, i.e.,
\begin{equation*}
\pi:\  Tri(\mathbf{R}^3) \rightarrow H,\quad  \pi(PQR):=\pi(P)\pi(Q)\pi(R),
\end{equation*}
where $\pi(P)\pi(Q)\pi(R)$ is the triangle defined by three vertices $\pi(P)$, $\pi(Q)$, and $\pi(R)$. By projecting slant triangles in $Tri(\mathbf{R}^3)$ on $H$, we obtain a triangular mesh $H_0$ on $H$, i.e.,
\begin{align*}
Tri(H_0)&:=\{\pi(PQR) \subset H \ | \ PQR \in Tri(\mathbf{R}^3) \}, \\
Edge(H_0)&:= \{ ab \subset H\  |\  \exists c \in Ver(H_0) \text{ such that } abc \in Tri(H_0) \},\\
Ver(H_0)&:=\{ \pi(v) \in H \ |\ v \in \mathbf{Z}^3 \}.
\end{align*}
The elements of $Tri(H_0)$ are called \textit{flat} triangles. In the following, \underline{we identify} \underline{$H_0$ with $M_0$}.
\end{definition}

\begin{remark}
Points of $\mathbf{R}^3$ are denoted by uppercase letters, whereas vertices of meshes are denoted by lowercase letters.
\end{remark}

Given $abc \in Tri(H_0)$. The ``tangent space’’ $T(abc)$ on $abc$ is defined by identifying slant triangles in $\pi^{-1}(abc)$ with the same slope.

\begin{definition}[$\sim_T$]
Given $P_1P_2P_3,\ Q_1Q_2Q_3 \in \pi^{-1}(abc)$. An equivelence relation $\sim_T$ on $ \pi^{-1}(abc)$ is defined by
\begin{equation*}
P_1P_2P_3 \sim_T Q_1Q_2Q_3 \text{ if and only if they have the same slope},
\end{equation*}
i.e., $\exists k \in \mathbf{Z}$ such that  
\begin{equation*}
\{Q_1, Q_2, Q_3\}=\{P_1+(k,\ k,\ k), P_2+(k,\ k,\ k), P_3+(k,\ k,\ k)\} \subset \mathbf{Z}^3.
\end{equation*}
\end{definition}

\begin{definition}[$T(abc)$]\label{def_T}
Given $abc \in Tri(H_0)$. \textit{Tangent space $T(abc)$ on $abc$} is the quotient set of $\pi^{-1}(abc)$ by $\sim_T$, i.e,
\begin{equation*}
T(abc):= \pi^{-1}(abc)/\sim_T,
\end{equation*}
(Figure \ref{figure5} (b) and (c)). Projection $\pi_T:\ T(abc) \rightarrow H_0$ is defined by
\begin{equation*}
\pi_T(abc \bmod {\sim}_T):=\pi(abc).
\end{equation*}
\end{definition}

\begin{definition} [$T(M)$ and $\Gamma(M)$]
Given $M \subset H_0$. \textit{Tangent bundle $T(M)$ on $M$} is the collection of all tangent spaces $T(abc)$ on $M$, i.e.,
\begin{equation*}
T(M):=\{T(abc)\ |\  abc \in Tri(M)\}.
\end{equation*}
A mapping $\sigma:\ M \rightarrow T(M)$ is called a \textit{section of $T(M)$ on $M$} if
\begin{equation*} 
\pi_T(\sigma(abc))=abc \quad \text{for } \forall abc \in Tri(M).
\end{equation*}
$\Gamma(M)$ denotes the collection of all sections of $T(M)$ on $M$, i.e.,
\begin{equation*} 
\Gamma(M):=\{\sigma \in MAP(M,T(M)) \ | \ \sigma \text{ is a section of $T(M)$ on $M$} \},
\end{equation*}
where $MAP(M,T(M))$ denotes the collection of all mappings from $M$ to $T(M)$. 
\end{definition}

\paragraph{3.2.4.2 \textit{Vector Fields on $M$ defined by Sections of $T(M)$}}\label{para3241}  
Given $M \subset H_0$. Elements of $C^1(M)$ (i.e., vector fields on $M$) are defined by ``edge-continuous'' sections of $T(M)$. ``Edge-continuous'' sections are defined using ``edge vectors'' $e_s(ab)$, $e_s(bc)$, and $e_s(ca)$.

\begin{definition}[$e_s(ab)$, $e_s(bc)$, and $e_s(ca)$]\label{def_e_vec}
Given $M \subset H_0$, $abc \in Tri(M)$, and $s \in T(abc)$. By definition, $\exists PQR \in \pi^{-1}(abc)$ such that 
\begin{equation*}
s = PQR \bmod {\sim}_T,
\end{equation*}
where $\pi(P)=a$, $\pi(Q)=b$, and $\pi(R)=c$. The \textit{edge vectors} $e_s(ab)$, $e_s(bc)$, and $e_s(ca)$ of $s$ are defined by
\begin{align*}
e_s(ab)&:= Q-P \in \mathbf{Z}^3, \\
e_s(bc)&:=R-Q \in \mathbf{Z}^3, \\
e_s(ca)&:=P-R \in \mathbf{Z}^3.
\end{align*}
Note that these values do not depend on the choice of $PQR  \in \pi^{-1}(abc)$.
\end{definition}

\begin{definition} [Edge-continuous sections] \label{def_edge_cont}
Given $M \subset H_0$ and $\sigma \in \Gamma (M)$, $\sigma$ is called \textit{edge-continuous} if 
\begin{equation*}
 e_{\sigma(abc)}(ab)= e_{\sigma(abd)}(ab)
\end{equation*}
for any pair $abc$, $abd \in Tri(M)$ sharing an edge $ab$.
\end{definition}

\begin{lemma} [Vector field $\beta_{\sigma}$]\label{lemma_A}
Given $M \subset H_0$ and edge-continuous $\sigma \in \Gamma(M)$. Then, a vector field $\beta_{\sigma} \in C^1(M)$ is obtained by
\begin{equation*}
\beta_{\sigma}(ab):= \Vert  e_{\sigma(abc)}(ab) \Vert \quad (ab \in C_1(M))),
\end{equation*}
where $abc \in Tri(M)$ is any triangle containing $ab$ as a side, and
\begin{equation*}
\Vert (x,y,z) \Vert :=x+y+z. 
\end{equation*}
 Since $\sigma$ is edge-continuous, the value dose not depend on the choice of $abc$. $\beta_{\sigma}$ is called the \textit{vector field defined by $\sigma$}.
\end{lemma}

\begin{proof}
It follows immediately from the definitions.
\end{proof}

\begin{example}
In Figure \ref{figure5}, $abc=\pi(ABO)$, where 
\begin{equation*}
A=(l+1,\ m,\ n),\ B=(l+1,\ m+1,\ n), \text{ and } O=(l,\ m,\ n).
\end{equation*}
Suppose $\sigma(abc)=ABO$, where $a=\pi(A)$, $b=\pi(B)$, and $c=\pi(O)$.  Then,
\begin{align*}
&\beta_{\sigma}(ab)= \Vert  e_{\sigma(abc)}(ab) \Vert =\Vert  B-A \Vert = \Vert  (0,\ 1,\ 0) \Vert =1, \\
&\beta_{\sigma}(bc)= \Vert  e_{\sigma(abc)}(bc) \Vert =\Vert  O-B \Vert = \Vert  (-1,\ -1,\ 0) \Vert =-2, \\
&\beta_{\sigma}(ca)= \Vert  e_{\sigma(abc)}(ca) \Vert =\Vert  A-O \Vert = \Vert  (1,\ 0,\ 0) \Vert =1.
\end{align*}
Note that $\beta_{\sigma}(ba)=-1$, $\beta_{\sigma}(cb)=2$, and $\beta_{\sigma}(ac)=-1$
\end{example}

\paragraph{3.2.4.3 \textit{Flows on $M$ defined by Vector Fields on $M$}}\label{para3243} 
Given $M \subset H_0$. Flows on $M$ are defined by ``consistent'' vector fields on $M$. 
\begin{remark}
We denotes the absolute value of $m \in \mathbf{Z}$ by $|m|$. For example, $|-3|=3$.
\end{remark}

\begin{definition} [Consistent vector fields]
Given $M \subset H_0$ and $\beta \in C^1(M)$. $\beta$ is called \textit{consistent} if 
\begin{equation*}
|\beta(ab)| \geq |\beta(bd)|,\   |\beta(da)| \quad \text{if } |\beta(ab)| \geq |\beta(bc)|,\  |\beta(ca)| 
\end{equation*}
for any pair $abc, abd \in Tri(S)$ sharing edge $ab$.
\end{definition}

\begin{lemma}[Flow $\Psi_{\beta}$]\label{lemma_B}
Given $M \subset H_0$, consistent $\beta \in C^1(M)$, and $abc \in Tri(M)$. Define the normal edges of $abc$ by
\begin{equation*}
ab \text{ is a normal edge of } abc \text{ if } |\beta(ab)| \geq  |\beta(bc)|,\  |\beta(ca)|.
\end{equation*}
Then, a flow $\Psi_{\beta} \in FLOW(M)$ is obtained by connecting triangles to adjacent triangles
\begin{equation*}
\text{through common edges that are not normal edges}. 
\end{equation*}
$\Psi_{\beta}$ is called the \textit{flow on $M$ defined by $\beta$}.
\end{lemma}

\begin{proof}
It follows immediately from the definitions.
\end{proof}

\begin{remark}
Given $abc,\ abd \in Tri(M)$. Since $\beta \in C^1(M)$ is consistent, $ab$ is a normal edge of $abc$ if and only if $ab$ is a normal edge of $abd$.
\end{remark}

\subsubsection{Computation of Flows on $M$ by Tangent Cones}\label{sec325}\    

Given $M \subset H_0$. We define ``tangent cones'' to compute elements of $\Gamma(M)$, $C^1(M)$, and $FLOW(M)$.

\paragraph{3.2.5.1 \textit{Tangent Cones}}\label{para3251} 
Given $M \subset H_0$. Recall that tangent bundle $T(M)$ is defined using unit cubes (Figure \ref{figure5} (a)). By stacking unit cubes diagonally from $(\infty, \infty, \infty)$ to $(-\infty, -\infty, -\infty)$, we obtain a triangular pyramid of infinite height with multiple tops and no base. In this paper, \underline{``tangent cones” refer to the} \underline{``triangular pyramids” obtained in this way, i.e., by cube stacking}. 

\begin{definition} [$Cone\ A$]
Given $A \subset \mathbf{Z}^3$ such that $A$ consists of finite elements. \textit{Tangent cone} $Cone\ A$ is the triangular pyramid generated by $A$ in $\mathbf{R}^3$, i.e.,
\begin{multline*}
Cone\ A:=\{(x, y, z) \in \mathbf{R}^3\ |\ \exists (u,v,w) \in A \\
\text{such that } x-u,\ y-v ,\ z-w >0 \} \subset \mathbf{R}^3.
\end{multline*}
Tangent cones may have multiple tops but no base. $tops(Cone\ A) \subset \mathbf{Z}^3$ denotes the tops of $Cone\ A$. ${\partial}Cone\ A \subset \mathbf{R}^3$ denotes the collection of all points on the surfaces of $Cone\ A$ i.e, 
\begin{multline*}
{\partial}Cone\ A :=\{ (x, y, z) \in \mathbf{R}^3\  |\ \\
max_{(u,v,w) \in A}\{min\{x-u,\ y-v,\ z-w\} \}=0 \}  \subset \mathbf{R}^3.
\end{multline*}
$CONE(H_0)$ denotes the collection of all tangent cones in $\mathbf{R}^3$, i.e.,
\begin{equation*}
CONE(H_0):=\{ Cone\ A\ |\ A \subset \mathbf{Z}^3, \text{ $A$ consists of finite elements}. \}.
\end{equation*}
\end{definition}

\begin{definition} [$v_{{\partial}Cone\ A}(x)$]\label{def_v_partial}
Given $Cone\ A \in CONE(H_0)$. A one-to-one mapping
\begin{equation*}
v_{{\partial}Cone\ A} :\ H_0 \rightarrow {\partial}Cone\ A
\end{equation*}
is uniquely defined by the equation
\begin{equation*}
{\partial}Cone\ A \cap \pi^{-1}(x) = \{ v_{{\partial}Cone\ A}(x) \}.
\end{equation*}
In particular, ${\pi}(v_{{\partial}Cone\ A}(x))=x$. $v_{{\partial}Cone\ A}(x)$ is called the \textit{surface map} of $Cone\ A$. In the following, \underline{we often identify $Cone\ A$ with $v_{{\partial}Cone\ A}(x)$}.
\end{definition}

\paragraph{3.2.5.2 \textit{Computation of Sections of $T(M)$}}\label{para3252} 
Given $M \subset H_0$. Elements of $\Gamma(M)$ are computed using tangent cones.

\begin{lemma}[Section $\sigma_{Cone\ A}$]\label{lemma_C}
Given $Cone\ A \in CONE(H_0)$. A section $\sigma_{Cone\ A} \in \Gamma(H_0)$ is obtained by
\begin{equation*}
\sigma_{Cone\ A}(abc) :=PQR\ \bmod \sim_T,
\end{equation*}
where
\begin{align*}
P&=v_{{\partial}Cone\ A}(a) \in \pi^{-1}(a),   \\
Q&=v_{{\partial}Cone\ A} (b) \in \pi^{-1}(b),  \\
R&=v_{{\partial}Cone\ A} (c) \in \pi^{-1}(c).
\end{align*}
$\sigma_{Cone\ A}$ is called the \textit{section of $T(H_0)$ defined by $Cone\ A$}. 
\end{lemma}
\begin{proof}
It follows immediately from the definitions.
\end{proof}

\begin{corollary}\label{cor_C2}
Given $Cone\ A \in CONE(H_0)$. Then, ${\sigma}_{Cone\ A}$ is edge-continuous.
\end{corollary}
\begin{proof}
Note that $PQR$ is on a surface of $Cone\ A$. The result follows immediately.
\end{proof}

\paragraph{3.2.5.3 \textit{Computation of Vector Fields on $M$}}\label{para3253} 
Given $M \subset H_0$. Elements of $C^0(M)$ and $C^1(M)$ are computed using tangent cones.

\begin{lemma}[Potential function $\alpha_{Cone\ A}$]\label{lem_C3}
Given $Cone\ A \in CONE(H_0)$. A potential function $\alpha_{Cone\ A} \in C^0(H_0)$ is obtained by
\begin{equation*}
\alpha_{Cone\ A}(a):=\Vert v_{{\partial}Cone\ A}(a) \Vert \quad (a \in C_0(H_0)),
\end{equation*}
where $\Vert (x,y,z) \Vert :=x+y+z$. ${\alpha}_{Cone\ A}$ is called the \textit{potential function on $H_0$ defined by $Cone\ A$}. 
\end{lemma}
\begin{proof}
It follows immediately from the definitions.
\end{proof}

\begin{lemma}[Vectot field $\beta_{Cone\ A}$]\label{lemma_D}
Given $Cone\ A \in CONE(H_0)$. A vector field $\beta_{Cone\ A} \in C^1(H_0)$ is obtained by
\begin{equation*}
\beta_{Cone\ A}(ab):=\beta_{\sigma_{Cone\ A} }(ab)= \Vert e_{\sigma_{Cone\ A}(abc)}(ab) \Vert \quad (ab \in C_1(H_0)),
\end{equation*}
where $abc \in Tri(H_0)$ is any triangle containing $ab$ as a side (Lemma \ref{lemma_A}). ${\beta}_{Cone\ A}$ is called the \textit{vector field on $H_0$ defined by $Cone\ A$}. 
\end{lemma}
\begin{proof}
By Corollary \ref{cor_C2},  ${\sigma}_{Cone\ A}$ gives an edge-continuous section on $H_0$. By Lemma \ref{lemma_A}, a vector field is defined by an edge-continuous section.
\end{proof}
\begin{remark}
For simplicity, we write ${\beta}_{Cone\ A}$ instead of ${\beta}_{{\sigma}_{Cone\ A}(abc)}$.
\end{remark}

\begin{corollary}\label{cor_D2}
Given $Cone\ A \in CONE(H_0)$. Then, $\beta_{Cone\ A}$ is consistent.
\end{corollary}
\begin{proof}
From the definitions, we have
\begin{enumerate}
\item \label{item:first} $\beta_{Cone\ A} (ab)\neq 0$ $\quad$ for $\forall ab \in Edge(H_0)$, 
\item \label{item:second} $\beta_{Cone\ A} (ab)+\beta_{Cone\ A} (bc)+\beta_{Cone\ A} (ca)=0$ $\quad$ for $\forall abc \in Tri(H_0)$,
\item \label{item:third} $\exists m \in \mathbf{Z}$ such that
\begin{equation*}
max\{|\beta_{Cone\ A} (ab)|,\ |\beta_{Cone\ A} (bc)|,\ |\beta_{Cone\ A} (ca)|\}=m 
\end{equation*}
for $\forall abc \in Tri(H_0)$.
\end{enumerate}
Given $ab \in Edge(H_0)$. Then, $ab$ is a normal edge if and only if 
\begin{equation*}
|\beta_{Cone\ A}(ab)|=m.
\end{equation*}
Because of (\ref{item:first}) and (\ref{item:second}), 
\begin{equation*}
\text{every triangle has only one normal edge}.
\end{equation*}
Given $abc$, $bcd \in Tri(H_0)$ such that $|\beta(bc)|=m$. Because of (\ref{item:third}), 
\begin{equation*}
\text{$bc$ is the normal edge of both $abc$ and $bcd$}. 
\end{equation*}
That is, $\beta_{Cone\ A}$ is consistent.
\end{proof}

\paragraph{3.2.5.4 \textit{Computation of Flows on $M$}}\label{para3254}   
Given $M \subset H_0$. Flows on $M$ are computed using tangent cones..

\begin{lemma}[Flow ${\Psi}_{Cone\ A}$]\label{lemma_E}
Given $Cone\ A \in CONE(H_0)$. Then, a flow $\Psi_{Cone\ A} \in FLOW(H_0)$ is obtained by
\begin{equation*}
\Psi_{Cone\ A}:=\Psi_{\beta_{Cone\ A}}
\end{equation*}
(Lemma \ref{lemma_B}). ${\Psi}_{Cone\ A}$ is regular on $H_0$. $\Psi_{Cone\ A}$ is called the \textit{flow on $H_0$ defined by $Cone\ A$}. 
\end{lemma}

\begin{proof}
By Corollary \ref{cor_D2}, $\beta_{Cone\ A}$ gives a consistent vector field on $H_0$. 
By Lemma \ref{lemma_B}.  a flow is defined by a consistent vwctor field. Since every triangle has only one normal edge, $\beta_{Cone\ A}$ is regular on $H_0$.
\end{proof}

\begin{remark}
For simplicity, we write $\Psi_{Cone\ A}$ instead of $\Psi_{\beta_{Cone\ A}}$.
\end{remark}

\begin{corollary}[Flows on $M$]\label{cor_flow_on_M}
Given $M \subset H_0$ and $Cone\ A \in CONE(H_0)$. Then, ${\Psi}_{Cone\ A}|_M \in FLOW(M)$..
\end{corollary}
\begin{proof}
It follows immediately from the definitions.
\end{proof}

\begin{figure}
\centering
\captionsetup{width=1.0\linewidth}
\includegraphics{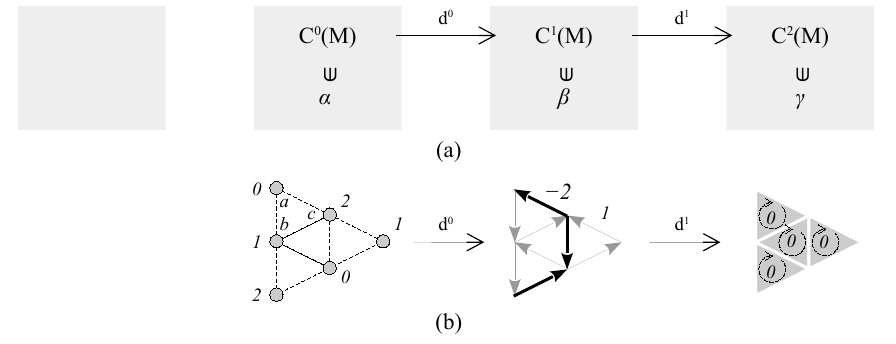}
\caption{Cochain complex $C^{\ast}(M)$.}
\label{figure6}
\end{figure}

\subsubsection{Cochain complex $C^{\ast}(M)$ on $M$ }\label{sec326}\ 

In cohomology theory, ``cochain complexes'' are the fundamental machinery used to define the cohomology class of vector fields on a geometric object. Here we define the cohomology class of vector fields on a mesh to identify the singularities of a loop. 

\begin{remark}
In this paper, we consider two types of complexes. One is a chemical complex, such as a molecule complex, and the other is a mathematical complex, such as a cochaine complex. Be careful not to confuse them.
\end{remark}

\begin{definition} [$C^{\ast}(M)$]
Given $M \subset H_0$. The \textit{cochain complex} $C^{\ast}(M)$ on $M$ is a collection of three objects $C^0(M)$, $C^1(M)$, and $C^2(M)$ with two linear mappings $d^1$ and $d^2$ defined by
\begin{align*}
d^0: C^0(M) &\rightarrow C^1(M),\quad  d^0f(v_0v_1):=f(v_1)-f(v_0), \\
d^1: C^1(M) &\rightarrow C^2(M),\quad  d^1f(v_0v_1v_2):=f(v_0v_1)+f(v_1v_2)+f(v_2v_0)
\end{align*}
(Figure \ref{figure6} (a)), $d^0$ and $d^1$ are called \textit{coboundary operators} of $C^{\ast}(M)$.
\end{definition}

\begin{remark}
$d^1d^0f=0$ for $\forall f \in C^0(M)$.
\end{remark}

\begin{example} 
In Figure \ref{figure6} (b) left, a potential function $\alpha \in C^0(M)$ is given: 
\begin{equation*}
\alpha(a)=0,\  \alpha(b)=1,\ \alpha(c)=2,\ \dots
\end{equation*}
Then, 
\begin{equation*}
d^0\alpha(ab)=1,\ d^0\alpha(bc)=1, \ d^0\alpha(ca)=-2,\ \dots
\end{equation*}
(Figure \ref{figure6} (b) middle), and
\begin{equation*}
d^1d^0\alpha(abc)=0,\  \dots
\end{equation*}
(Figure \ref{figure6} (b) right).
\end{example}

\begin{definition} [Exactness and closedness]
Given $M \subset H_0$ and $\beta \in C^1(M)$. $\beta$ is called \textit{exact} if $\exists \alpha \in C^0(M)$ such that
\begin{equation*}
\beta=d^0\alpha. 
\end{equation*}
$\beta$ is called \textit{closed} if 
\begin{equation*}
d^1\beta=0.
\end{equation*}
\end{definition}

\begin{remark}
Exact vector fields are closed. 
\end{remark}

\begin{example}
The vector field given in Figure \ref{figure6} (b) middle is exact and closed.
\end{example}

\begin{lemma}\label{lem_b_exact}
Given $Cone\ A \in CONE(H_0)$. Then, $\beta_{Cone\ A} \in C^1(H_0)$ is exact, i.e.,
\begin{equation*}
\beta_{Cone\ A}=d^0\alpha_{Cone\ A}.
\end{equation*}
\end{lemma}
\begin{proof}
It follows immediately from the definitions.
\end{proof}

\begin{lemma}\label{lem_b_closed}
Given $M \subset H_0$ and edge-continuous  $\sigma \in \Gamma(M)$,Then, $\beta_{\sigma} \in C^1(M)$ is closed, i.e.,
\begin{equation*}
d^1\beta_{\sigma}=0.
\end{equation*}
\end{lemma}
\begin{proof}
It follows immediately from the definitions.
\end{proof}

\begin{definition} [$H^1(M)$]
Given $M \subset H_0$. The \textit{first cohomology set} $H^1(M)$ on $M$ is the quotient set defined by
\begin{equation*}
H^1(M):=Ker_M(d^1)/Image_M(d^{0}).
\end{equation*}
where
\begin{align*}
Image_M(d^0)&:=\{\beta \in C^1(M)\  |\  \exists \alpha \in C^0(M) \text{ such that } \beta=d^0\alpha \} \subset C^1(M), \\
Ker_M(d^1)&:=\{\beta \in C^1(M)\  |\  d^1\beta =0 \} \subset C^1(M).
\end{align*}
The \textit{cohomology class} $[\beta]$ of $\beta$ is the image of $\beta \in Ker_M(d^1)$ in $H^1(M)$,
\end{definition}

\begin{remark}
$Image_M(d^{0}) \subset Ker_M(d^1)$ because $d^1d^0f=0$ for $\forall f \in C^0(M)$.
\end{remark}

\begin{remark}
We write $H^1(M)=0$ if $Ker_M(d^1)=Image_M(d^{0})$.
\end{remark}

\begin{proposition}
$H^1(M_0)=0$.
\end{proposition}
\begin{proof}
It is sufficient to show that $Image(d^0) \supset Ker(d^1)$. Given $\beta \in C^1(M_0)$ such that 
\begin{equation*}
d^1\beta=0.
\end{equation*}
We can construct $\alpha \in C^0(M_0)$ such that 
\begin{equation*}
\beta=d^0\alpha
\end{equation*}
as follows:

Choose any vertex $v_0 \in Ver(M_0)$, and define $\alpha(v_0)$ by
\begin{equation*}
\alpha(v_0):=0.
\end{equation*}
For vertex $v_1 \in Ver(M_0)$ other than $v_0$, define $\alpha([v_1])$ by
\begin{equation*}
\alpha(v_1):=\beta(v_0v_{a1})+\beta(v_{a1}v_{a2})+ \cdots +\beta(v_{a(k-1)}v_{ak}) + \beta(v_{ak}v_1),
\end{equation*}
where 
\begin{equation*}
v_0v_{a1} \cdots v_{ak}v_1 \text{ is a polygonal line connecting $v_0$ and $v_1$}.
\end{equation*}
Note that the value on the right side does not depend on the polygonal line connecting $v_0$ and $v_1$ because $d^1\beta=0$.
\end{proof}

\subsubsection{Cone complex $\prod^{\ast} CONE(\bigcap^{\ast}S_i)$ on $M$}\label{sec327}\ 

Given $M \subset H_0$ and  $\beta \in C^1(M)$. We define another complex on $M$ to compute an ``continuation'' of $\beta$ over the holes within $|M|$. (In other words, the domain of $\beta$ is extended over the holes within $|M|$.)

\paragraph{3.2.7.1 \textit{Affine Flows and Collective Affine Flows}}\label{para3271} 
Flows of triangles are broadly classified into two kinds: one is defined by a tangent cone, the other is defined by multiple tangent cones collectively.

Given $M \subset H_0$. Identifying a tangent cone $Cone\ A$ with its surface map $v_{{\partial}Cone\ A}(x)$ (Definition \ref{def_v_partial}), we can consider the restriction of $Cone\ A$ on $M$.

\begin{figure}
\centering
\captionsetup{width=1.0\linewidth}
\includegraphics{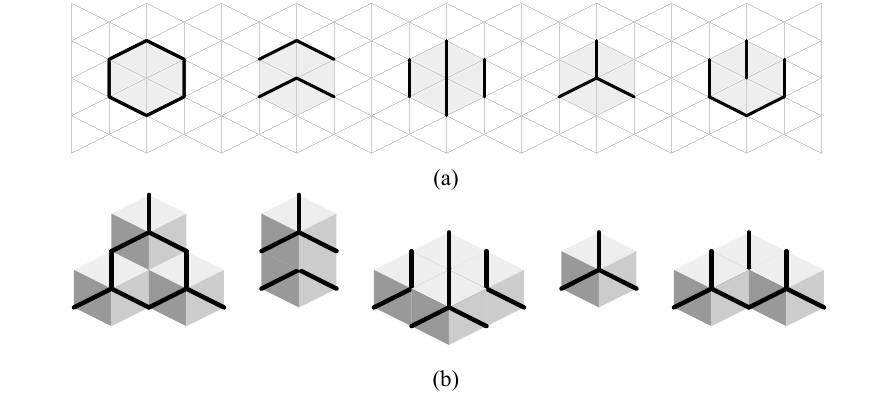}
\caption{Flows on a hexagon consisting of $6$ triangles.}
\label{figure7}
\end{figure}

\begin{definition} [$CONE(M)$]
Given $M \subset H_0$ and $Cone\ A \in CONE(H_0)$. $(Cone\ A, M)$ denotes the restriction of $Cone\ A$ on $M$, i.e., 
\begin{equation*}
(Cone\ A, M) := \text{the surface map $v_{{\partial}Cone\ A}(x)$ restricted on $|M|$}. 
\end{equation*}
$CONE(M)$ is the collection of all tangent cones restricted on $M$, i.e.,
\begin{equation*}
CONE(M):=\{(Cone\ A, M)\ |\ Cone\ A \in CONE(H_0)\}.
\end{equation*}
\end{definition}

\begin{definition} [Affine flow $\Psi(Cone\ A, M)$]
Given $M \subset H_0$ and $Cone\ A \in CONE(H_0)$. $\Psi(Cone\ A\, M)$ is the regular flow on $M$ defined by
\begin{equation*}
\Psi(Cone\ A, M) := \Psi_{Cone\ A}|_{M} \quad  \in FLOW(M)
\end{equation*}
(Corollary \ref{cor_flow_on_M}). $\Psi(Cone\ A, M)$ is called an \textit{affine} flow on $M$. $Cone\ A$ is called the \textit{tangent cone associated with the flow}. 
\end{definition}

\begin{remark}
All affine flows on $M \subset H_0$ can be extended to affine flows on $H_0$. 
\end{remark}

\begin{example}[Flows on a hexagon] \label{example_hex}
Figure \ref{figure7} (a) shows all regular flows on a hexagonal region on $H_0$ consisting of $6$ triangles (excluding rotations). All of them are affine (Figure \ref{figure7} (b)).
\end{example}

\begin{definition} [Collective affine flow $\sum_{i \in I} \Psi(Cone\ A_i, S_i)$]
Given $M \subset H_0$, its open covering $V=\{S_i \subset M \ |\ i \in I \}$ (Definition \ref{def_covering}), and locally defined affine flows 
\begin{equation*}
CAF=\{\Psi(Cone\ A_i, S_i) \in FLOW(S_i) \ |\ i \in I\}.
\end{equation*}
Suppose $CAF$ is consistent (Definition \ref{def_consistancy21}), i.e.,
\begin{equation*}
\Psi(Cone\ A_i, S_i)|_{S_i \cap S_j }= \Psi(Cone\ A_j, S_j)|_{S_i \cap S_j } \quad (i,\ j \in I).
\end{equation*}
Then, by patching the affine flows in $CAF$ together, we obtain a flow $\Psi_0 \in FLOW(M)$ (Lemma \ref{lem_colletiveflows}), i.e.,
\begin{equation*}
\Psi_{0}|_{S_i} := \Psi(Cone\ A_i, S_i) \quad \in FLOW(S_i) \quad (i \in I).
\end{equation*}
$\Psi_0$ is called the \textit{collective affine flow on $M$ (defined by $CAF$)} and denoted as a formal sum of locally defined affine flows, i.e.,
\begin{equation*}
\Psi_0=\sum_ {i \in I} \Psi(Cone\ A_i, S_i) \quad \in FLOW(M).
\end{equation*}
\end{definition}

\begin{remark}
Different tangent cones may define the same flow on a mesh. Therefore, 
\begin{equation*}
\Psi(Cone\ A_i, S_i)= \Psi(Cone\ A_j, S_j) \quad \text{on } S_i \cap S_j
\end{equation*}
(i.e., two flows are identical on $S_i \cap S_j$) dose not imply 
\begin{equation*}
(Cone\ A_i, S_i)= (Cone\ A_j, S_j) \quad \text{on } S_i \cap S_j
\end{equation*}
(i.e., two tangent cones are identical on $S_i \cap S_j$) (see Lemma \ref{lemma_H} below).
\end{remark}

\begin{figure}
\centering
\captionsetup{width=1.0\linewidth}
\includegraphics{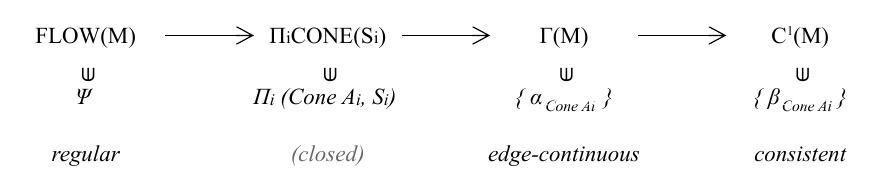}
\caption{Computation of $\{\beta_{Cone\ A_i}\} \subset C^1(M)$ ($i \in I$) from $\Psi \in FLOW(M)$. (The definition of ``closed'' locally defined tangenmt cones is given in Definition \ref{def_exact_and_closed_cones}.)}
\label{figure9}
\end{figure}

\begin{proposition} \label{proposition_F}
Given $M \subset H_0$. All regular flows on $M$ are collective affine flows.
\end{proposition}

\begin{proof}
Let 
\begin{equation*}
V=\{ U_i \subset M\ |\ i \in I \} \quad (I \subset \mathbf{Z})
\end{equation*}
be the collection of all hexagonal regions on $M$ consisting of $6$ triangles. Then, $V$ gives an open covering of $M$. Let 
\begin{equation*}
\Psi_0 \in FLOW(M)
\end{equation*}
 be a regular flow on $M$. Since $\Psi_0$ is affine on $U_i$ (Example \ref{example_hex}), $\exists Cone\ A_i \in CONE(H_0)$ such that
\begin{equation*}
\Psi_0|_{U_i}=\Psi(Cone\ A_i, U_i) \quad (i \in I). 
\end{equation*}
Then,
\begin{equation*}
\Psi(Cone\ A_i, U_i)|_{U_i \cap U_j}=\Psi(Cone\ A_j, U_j)|_{U_i \cap U_j} \quad (i,\j \in I),
\end{equation*}
and 
\begin{equation*}
\Psi_0=\sum_i \Psi(Cone\ A_i, U_i).
\end{equation*}
\end{proof}

\begin{remark}
$M$ may have holes inside.
\end{remark}

\begin{proposition}\label{pro_G}
Given $M \subset H_0$ and a regular flow $\Psi_0 \in FLOW(M)$. Then, 
\begin{equation*}
\exists \beta \in C^1(M) \text{ such that }\Psi_0={\Psi}_{\beta}
\end{equation*}
(Figure \ref{figure9}). Moreover, 
\begin{equation*}
[\beta]=0 \in H^1(M) \quad \text{if $\Psi_0$ is affine.}
\end{equation*}
\end{proposition}

\begin{proof} 
By Proposition \ref{proposition_F},  $\exists$an open covering $V=\{S_i \subset M\ |\ i \in I \}$ of $M$ and 
locally defined tangent cones 
\begin{equation*}
\{(Cone\ A_i, S_i) \in CONE(S_i) \ |\ i \in I \} 
\end{equation*}
such that
\begin{equation*}
\Psi_0=\sum_{i \in I} \Psi(Cone\ A_i, S_i).
\end{equation*}
Note that
\begin{equation*}
\Psi(Cone\ A_i, U_i)|_{U_i \cap U_j}=\Psi(Cone\ A_j, U_j)|_{U_i \cap U_j} \quad (i,j \in I).
\end{equation*}
By Lemma \ref{lemma_D} and Corollary \ref{cor_D2}, 
\begin{equation*}
\exists \text{ consistent } \beta_{Cone\ A_i} \in C^1(H_0) \quad (i \in I). 
\end{equation*}
Then,
\begin{equation*}
\beta_{Cone\ A_i}|_{U_i \cap U_j}=\beta_{Cone\ A_j}|_{U_i \cap U_j} \quad (i,j \in I),
\end{equation*}
where $\beta_{Cone\ A_i}|_{U_i \cap U_j}$ denotes the ristriction of $\beta_{Cone\ A_i}$ on $U_i \cap U_j$, i.e., 
\begin{equation*}
\beta_{Cone\ A_i}|_{U_i \cap U_j} \in C^1(U_i \cap U_j).
\end{equation*}
By Lemma \ref{lemma_E}, 
\begin{equation*}
\Psi_0|_{S_i}=\Psi_{\beta_{Cone\ A}}|_{S_i}  \quad (i \in I). 
\end{equation*}
Define $\beta \in C^1(M)$ by
\begin{equation*}
\beta|_{S_i}:=\beta_{Cone\ A_i} \quad (i \in I).
\end{equation*}
Then,
\begin{equation*}
\Psi_0={\Psi}_{\beta}.
\end{equation*}
Moreover, by lemma\ref{lem_b_exact},
\begin{equation*}
\beta|_{S_i}:=d^0\alpha_{Cone\ A_i} \quad (i \in I).
\end{equation*}
If $\Psi_0$ is affine, we have 
\begin{equation*}
\beta_{Cone\ A}=d^0\alpha_{Cone\ A}.
\end{equation*}
\end{proof}

\begin{remark} Note that
\begin{equation*}
\Psi(Cone\ A_i, U_i)|_{U_i \cap U_j}=\Psi(Cone\ A_j, U_j)|_{U_i \cap U_j} \quad (i,j \in I)
\end{equation*}
dose not imply
\begin{equation*}
\alpha_{Cone\ A_i}|_{U_i \cap U_j}=\alpha_{Cone\ A_j}|_{U_i \cap U_j} \quad (i,j \in I).
\end{equation*}
See Example \ref{example_Q} (The \textit{Penrose stairs}).
\end{remark}

I have no proof for the following claim:
\begin{claim}\label{claim3}
Regular flows on $H_0$ are affine.
\end{claim}

\begin{remark}
Given $\Psi \in FLOW(H_0)$. From Claim \ref{claim3}, it follows that 
\begin{equation*}
\Psi \text{ is regular on } H_0 
\end{equation*}
if and only if 
\begin{equation*}
\exists Cone\ A \in CONE(H_0) \text{ such that } \Psi=\Psi_{Cone\ A}.
\end{equation*}
\end{remark}

\begin{figure}
\centering
\captionsetup{width=1.0\linewidth}
\includegraphics{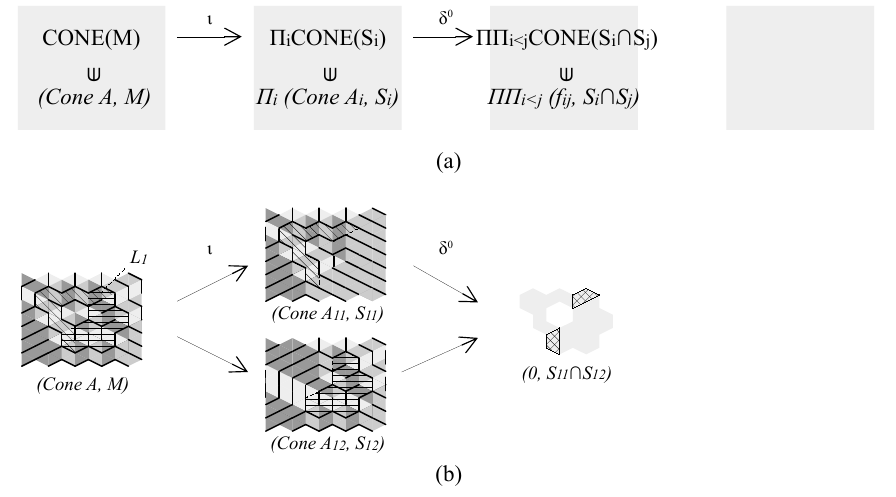}
\caption{Cone complex $\prod^{\ast} CONE(\bigcap^{\ast}S_i)$.}
\label{figure8}
\end{figure}

\paragraph{3.2.7.2 \textit{Cone Complex $\prod^{\ast} CONE(\bigcap^{\ast}S_i)$ on $M$}}\label{para3272} 
Given $M \subset H_0$ (with holes inside). We consider formal products of tangent cones to extend the domain of collective affine flows on $M$ over the holes within $|M|$.

\begin{definition} [$\prod_{i \in I} CONE(S_i)$]
Given $M \subset H_0$ and its open covering 
\begin{equation*}
V=\{S_i \subset M \ |\ i \in I \}.
\end{equation*}
 $\prod_{i \in I} CONE(S_i)$ is the collection of all formal products of tangent cones on $S_i \subset M$ ($i \in I$), i.e.,
\begin{equation*}
\prod_{i \in I} CONE(S_i):=\left\{ \prod_{i \in I} (Cone\ A_i, S_i)\  |\ (Cone\ A_i, S_i) \in CONE(S_i) \ (i \in I)\right\}.
\end{equation*}
\end{definition}

\begin{definition} [$\prod\prod_{i<j} CONE(S_i \cap S_j)$]
Given $M \subset H_0$ and its open covering 
\begin{equation*}
V=\{S_i \subset M \ |\ i \in I \}.
\end{equation*}
$\prod\prod_{i<j} CONE(S_i \cap S_j)$ is the collection of all formal products of $\mathbf{R}$-valued functions on $|S_i \cap S_j| \subset \mathbf{R}^2$ ($i, j \in I,\ i<j$), i.e.,
\begin{multline*}
\prod\prod_{i<j} CONE(S_i \cap S_j):=
\left\{ \prod_{i \in \{i \in I | i<j\}}\prod_{i \in I} (f_{i, j}, S_i \cap S_j)\ |\  \right.\\
\left.\phantom{\prod_{i \in I}} \text{$f_{i, j}$ is a $\mathbf{R}$-valued function on $|S_i \cap S_j|$  ($i, j \in I,\  i<j$)} \right\},
\end{multline*}
where we write $(f_{i, j}, S_i \cap S_j)$ instead of $f_{i, j}$ to emphasize its domain.
\end{definition}

\begin{definition} [Cone complex $\prod^{\ast} CONE(\bigcap^{\ast}S_i)$]
Given $M \subset H_0$ and its open covering $V=\{S_i \subset M \ |\ i \in I \}$. The \textit{cone complex} on $M$ is a collection of three objects $CONE(M)$, $\prod_{i \in I}CONE(S_i)$, and $\prod\prod_{i<j} CONE(S_i \cap S_j)$ with two linear mappings $\iota$ and $\delta_0$ defined by
\begin{equation*}
\iota:CONE(M) \rightarrow \prod_{i \in I}CONE(S_i),\quad
\iota\left((Cone\ A, M)\right):= \prod_{i \in I} \left(Cone\ A, S_i\right),
\end{equation*}
\begin{multline*}
\delta^0: \prod_{i \in I}CONE(S_i) \rightarrow \prod\prod_{i<j} CONE(S_i \cap S_j),\quad  \delta^0\left(\prod_{i \in I} (Cone\ A_i, S_i)\right):= \\
 \prod_{i \in \{i \in I|i<j\}}\prod_{i \in I} \left(v_{{\partial}Cone\ A_i}(x) - v_{{\partial}Cone\ A_j}(x), S_i \cap S_j\right)
\end{multline*}
(Figure \ref{figure8} (a)).
\end{definition}

\begin{remark}
$\delta^0 \iota(Cone\ A)=0$ for $\forall Cone\ A \in CONE(M)$.
\end{remark}

\begin{definition} [Exactness and closedness]\label{def_exact_and_closed_cones}
Given $M \subset H_0$, its open covering $V=\{S_i \subset M \ |\ i \in I \}$, and locally defined tangent cones 
\begin{equation*}
\prod_{i \in I} (Cone\ A_i, S_i) \in \prod_{i \in I} CONE(S_i). 
\end{equation*}
$\prod_{i \in I} (Cone\ A_i, S_i)$ is called \textit{exact} if $\exists (Cone\ A, M) \in CONE(M)$ such that
\begin{equation*}
\prod_{i \in I} (Cone\ A_i, S_i) =\iota((Cone\ A, M)). 
\end{equation*}
$\prod_{i \in I} (Cone\ A_i, S_i)$ is called \textit{closed} if 
\begin{equation*}
\delta^0\left(\prod_{i \in I} (Cone\ A_i, S_i)\right)=
\prod_{i \in \{i \in I|i<j\}}\prod_{i \in I} \left(0, S_i \cap S_j\right) \quad (=0).
\end{equation*}
\end{definition}

\begin{lemma}\label{lemma_H} 
Given $M \subset H_0$, its open covering $V=\{S_i \subset M\ |\ i \in I \}$, and locally defined tangent cones 
\begin{equation*}
\prod_{i \in I} (Cone\ A_i, S_i) \in \prod_{i \in I} CONE(S_i).
\end{equation*}
Then, 
\begin{equation*}
\left\{\Psi (Cone\ A_i, S_i)\ |\ i \in I \right\} \text{ is consistent  if } \prod_{i \in I} (Cone\ A_i, S_i) \text{ is closed}
\end{equation*}
(Definition \ref{def_consistancy21}). The opposite is not true.
\end{lemma}

\begin{proof}
Note that 
\begin{equation*}
\Psi(Cone\ A, H_0) = \Psi(Cone\ A_k, H_0),
\end{equation*}
where $A_k= \{a+(k,\ k,\ k)\ |\ a \in A\}$ ($k \in \mathbf{Z}$). That is, different tangent cones may define the same flow on $H_0$. The result follows immediately. 
\end{proof}

\begin{remark}
We can choose $\prod_{i \in I} (Cone\ A_i, S_i)$ in such a way that it is closed if $\sum_{i \in I} \Psi(Cone\ A_i, S_i)$ is an affine flow.
\end{remark}

\begin{example}
In Figure \ref{figure8} (Example \ref{ex_holes} and Example \ref{ex_17}),
\begin{equation*}
\Psi(Cone\ A, M)=\{L_1\} \quad \in FLOW(M).
\end{equation*}
$M$ is covered by an open covering $\{S_{11},\ S_{12}\}$ and 
\begin{align*}
&\iota(Cone\ A, M)= (Cone\ A_{11}, S_{11}) \times  (Cone\ A_{12}, S_{12}), \\
&\delta^0 ((Cone\ A_{11}, S_{11}) \times  (Cone\ A_{12}, S_{12}))=(0, S_{11}\cap S_{12}).
\end{align*}
That is, 
\begin{equation*}
(Cone\ A_{11}, S_{11}) \times  (Cone\ A_{12}, S_{12}) \text{ is closed}.
\end{equation*}
Then, 
\begin{equation*}
\left\{\Psi(Cone\ A_{11}, S_{11}),\ \Psi(Cone\ A_{12}, S_{12})\right\} \text{ is consistent}
\end{equation*}
and
\begin{equation*}
\Psi(Cone\ A, M) =\Psi(Cone\ A_{11}, S_{11})+ \Psi(Cone\ A_{12}, S_{12}),
\end{equation*}
where $A=A_{11}\cup A_{12}$. 
\end{example}

\subsubsection{Singularity analysis of loops}\label{sec328}\ 

Finally, we derive a determining equation for the singularities within a loop. The computation proceeds as follows (Figure \ref{figure10}).

\begin{figure}
\centering
\captionsetup{width=1.0\linewidth}
\includegraphics{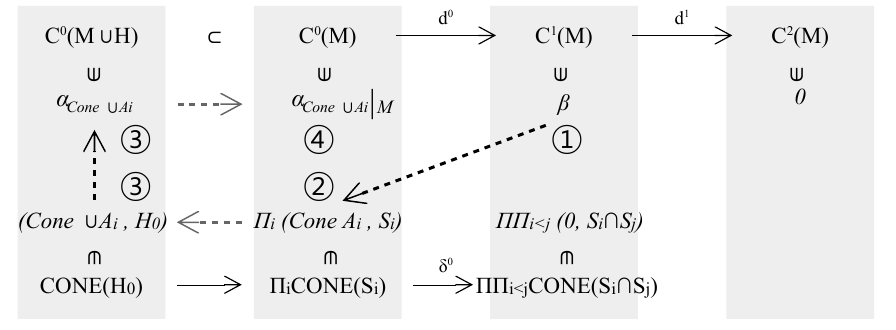}
\caption{Compution of $\alpha_{Cone\  \cup A_i} \in C^0(M \cup H)$ from $\beta \in C^1(M)$.}
\label{figure10}
\end{figure}

\begin{enumerate}
\item \textit{From a loop to a flow on $M$.}\\
Given $L \in LOOP(H_0)$. Define $M \subset H_0$ by
\begin{align*}
Tri(M)&:=\{abc \in Tri(H_0)\ |\  abc \in L\}. \\
Edge(M)&:=\{ ab \subset H\ |\  \exists c \in Ver(H_0) \text{ such that } abc \in Tri(M) \}, \\
Ver(M)&:=\{ a \in H \ |\  \exists b \in Ver(H_0) \text{ such that } ab \in Edge(M) \}.
\end{align*}
Then, 
\begin{align*}
\Psi_0:&=\{ L \} \in FLOW(M),\\
|L|_0&=|M|,
\end{align*}
and $\Psi_0$ is regular on $M$. By Proposition \ref{pro_G}, 
\begin{equation*}
\exists \beta \in C^1(M) \text{ such that } \Psi_0=\Psi_{\beta}
\end{equation*}
(Figure \ref{figure10} second from the right on the top row). Note that $|L|_0$ may have holes inside, i.e., $HOLE(L) \neq \emptyset$.

\item \textit{From a flow on $M$ to a collective affine flow on $M$.}\\
By Proposition \ref{proposition_F}, we have an open covering
\begin{equation*}
V=\{S_i \subset M\ |\ i \in I \}
\end{equation*}
of $M$ and locally defined tangent cones
\begin{equation*}
 \{(Cone\ A_i, S_i) \in CONE(S_i)\ |\ i \in I \}
\end{equation*}
such that
\begin{equation*}
\Psi_0=\sum_{i \in I} \Psi(Cone\ A_i, S_i)
\end{equation*}
(Figure \ref{figure10} second from left on the bottom row). Note that 
${\prod}_{i\in I}(Cone\ A_{i}, S_i)$ may not be closed, i.e., 
``$(CONE\ A_i, S_i)=(CONE\ A_j, S_j)$ on $S_i \cap S_j$'' may not hold true.

\item \textit{From a collective affine flow on $M$ to an affine flow on $M\cup H$}\\
Let $H \subset H_0$ be the submesh consisting of the holes within $L$, i.e.,
\begin{align*}
&Tri(H):= \left\{ abc \in Tri(H_0)\  |\  \exists U \in HOLE(L) \right. \\
&\left. \phantom{Tri(H):= \left\{ abc \in Tri(H_0)\  | \right\}}\quad \text{ such that } |abc| \subset U \right\}. 
\end{align*}
Then, 
\begin{equation*}
|L|=|M \cup H|.
\end{equation*}
Define $Cone\ A \in CONE(H_0)$ by
\begin{equation*}
Cone\ A := Cone\ \bigcup_{i \in I} A_i. 
\end{equation*}
By Lemma \ref{lemma_E}, a regular flow on $H_0$ is obtained by
\begin{equation*}
\Psi_{Cone\ A} \in FLOW(H_0).
\end{equation*}
In particular, 
\begin{equation*}
\Psi_{Cone\ A}|_{M \cup H} \in FLOW(M \cup H)
\end{equation*}
is regular (Figure \ref{figure10} left on the bottom row). 
By Lemma \ref{lem_C3}, we obtain
\begin{equation*}
\alpha_{Cone\ A} \in C^0(M \cup H)
\end{equation*}
(Figure \ref{figure10} left on the top row). 

\item \textit{The determing equation for the singularities within a loop.}\\
Since $\Psi_{Cone\ A}$ is regular on $M \cup H$, we obtrain a loop decomposition of the region $|H|$ if 
\begin{equation*}
\Psi_0=\Psi_{Cone\ A}|_{M} \quad \in FLOW(M)
\end{equation*}
or equivalently
\begin{equation*}
\beta=d^0(\alpha_{Cone\ A}|_{M}) \quad \in C^1(M), 
\end{equation*}
In other words, the regions of $HOLE(L)$ are not singular if the equation holds. 
\end{enumerate}

\begin{figure}
\centering
\captionsetup{width=1.0\linewidth}
\includegraphics{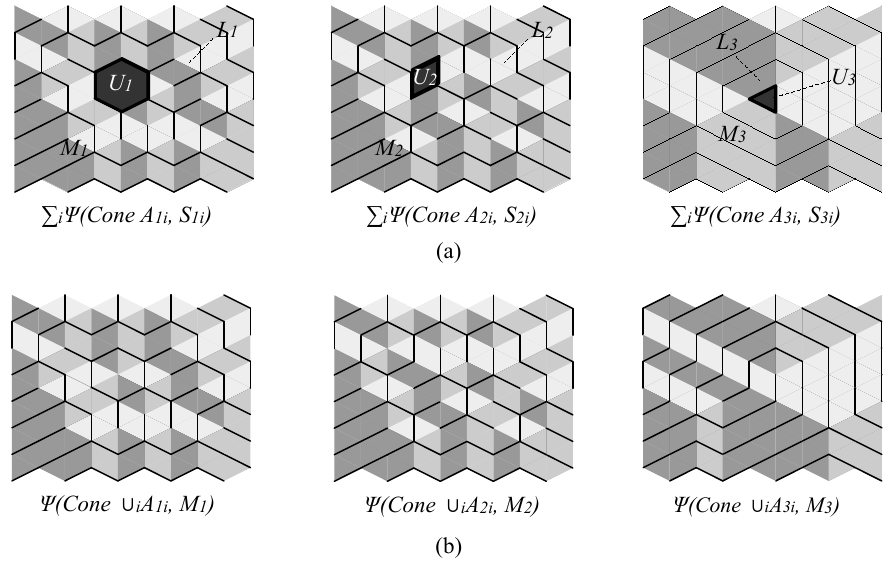}
\caption{Determing equation: (a) Left side $\sum_{i \in I} \Psi(Cone\ A_i, S_i)$. (b) Right side $\Psi_{Cone\ \bigcup_{i \in I} A_i}|_{M}$.}
\label{figure11}
\end{figure}

\begin{proposition}[Determing equation for the singularities within a flow]\label{prop_defeq}
Given $M \subset H_0$, its open covering $V=\{S_i \subset M\ |\ i \in I \}$, and locally defined affine flows 
\begin{equation*}
CAF=\{\Psi(Cone\ A_i, S_i) \in FLOW(S_i)\ |\ i \in I\}.
\end{equation*}
Suppose $CAF$ is consistent. Then, a flow $\Psi_0 \in FLOW(M)$ is obtained by
\begin{equation*}
\Psi_0:=\sum_{i \in I} \Psi(Cone\ A_i, S_i).
\end{equation*}
Equation (\ref{eq_det}) below is called the \textit{determing equation} for the singularities within $\Psi_0$:
\begin{equation}\label{eq_det}
\sum_{i \in I} \Psi(Cone\ A_i, S_i)=\Psi_{Cone\ \bigcup_{i \in I} A_i}|_{M}.
\end{equation}
Then, 
\begin{enumerate}
\item $HOLE(\Psi_0)$ contains no singular region if Equation (\ref{eq_det}) holds. 
\item $HOLE(\Psi_0)$ contains a singular region if ${\prod}_{i\in I}(Cone\ A_{1i}, S_i)$ is closed and Equation (\ref{eq_det}) dose not hold. 
\item $HOLE(\Psi_0)$ contains a singular region if we cannot choose ${\prod}_{i\in I}(Cone\ A_{1i}, S_i)$ in such a way that it is closed. 
\end{enumerate}
\end{proposition}

\begin{proof}
It follows immediately from the definitions.
\end{proof}

Figure \ref{figure11} shows the three loops of Figure \ref{figure4} (a). Each of them is explained below.

\begin{example}
In Figure \ref{figure4} (a) left (Example \ref{ex_17} and \ref{ex_18}), 
\begin{align*}
\Psi_0:&=\{L_1\} \quad \in FLOW(M_1), \\
HOLE(L_1)&=\{ U_1 \} \quad (U_1 \subset \mathbf{R}^2).
\end{align*}
By Proposition \ref{proposition_F}, $\exists$an open covering $V=\{S_i \subset M_1\ | \ i \in I\}$ of $M_1$ and locally define affine flows 
\begin{equation*}
\{\Psi(Cone\ A_i, S_i) \in FLOW(S_i)\ |\  i\in I\} 
\end{equation*}
such that
\begin{equation*}
\Psi_0|_{S_i}=\Psi(Cone\ A_i, S_i) \quad (i \in I)
\end{equation*}
(Figure \ref{figure11} (a) left). In this case, we can choose ${\prod}_{i\in I}(Cone\ A_{1i}, S_i)$ in such a way that it is closed, i.e.,
\begin{equation*}
(Cone\ A_{1i}, S_i)=(Cone\ A_{1j}, S_j) \quad \text{on } S_{1i} \cap S_{1j}
\end{equation*}
(because $U_1$ is not singular).
Define $Cone\ A \in CONE(H_0)$ by
\begin{equation*}
Cone\ A := Cone\ \bigcup_{i \in I} A_i 
\end{equation*}
(Figure \ref{figure11} (b) left). Then,  
\begin{equation*}
\Psi_0=\Psi(Cone\  A, M_1) 
\end{equation*}
and $U_1$ is decomposed into a loop of length $6$ by 
\begin{equation*}
\Psi(Cone\  A, M_{U_1}).
\end{equation*}
\end{example}

\begin{example}\label{example_P}
In Figure \ref{figure4} (a) middle (Example \ref{ex_19}), 
\begin{align*}
\Psi_0:&=\{L_2\} \quad \in FLOW(M_2), \\
HOLE(L_2)&=\{ U_2 \} \quad (U_2 \subset \mathbf{R}^2).
\end{align*}
Given an open covering $V=\{S_{21}, S_{22}\}$ of $M_2$ (Figure \ref{figure4} (b) middle). 
Then, 
\begin{equation*}
\exists \{\Psi(Cone\ A_{2i}, S_{2i}) \in FLOW(S_{2i})\ |\  i=1, 2\} 
\end{equation*}
such that 
\begin{equation*}
\prod_{i =1, 2} (Cone\ A_{2i}, S_{2i}) \text{ is closed} 
\end{equation*}
and
\begin{equation*}
\Psi_0|_{S_i}=\Psi(Cone\ A_{2i}, S_{2i}) \quad (i =1, 2).
\end{equation*}
(Figure \ref{figure11} (a) middle).
Define $Cone\ A \in CONE(H_0)$ by
\begin{equation*}
Cone\ A := Cone\ A_{21}\cup A_{22} 
\end{equation*}
(Figure \ref{figure11} (b) middle). Then, $Cone\  A$ gives a loop decomposition of $|M_2 \cup M_{U_2}|$, but
\begin{equation*}
\Psi_0 \neq \Psi(Cone\  A, M_2)
\end{equation*}
(Figure \ref{figure11} (b) left). Therefore, $U_2$ is singular. 
\end{example}

\begin{example}[The \textit{Penrose stairs}]\label{example_Q}
In Figure \ref{figure4} (a) right (Example \ref{ex_20}), 
\begin{align*}
\Psi_0:&=\{L_3\} \quad \in FLOW(M_3), \\
HOLE(L_3)&=\{ U_3 \} \quad (U_3 \subset \mathbf{R}^2).
\end{align*}
Given an open covering $V=\{S_{31}, S_{32}, S_{33}\}$ of $M_3$ (Figure \ref{figure4} (b) right). 
Then, 
\begin{equation*}
\exists \{\Psi(Cone\ A_{3i}\ S_{3i}) \in FLOW(S_{3i})\ |\  i=1, 2, 3\}
\end{equation*}
such that 
\begin{equation*}
\Psi_0|_{S_i}=\Psi(Cone\ A_{3i}, S_{3i}) \quad (i =1, 2, 3).
\end{equation*}
(Figure \ref{figure11} (a) right). In this case, we \underline{cannot} choose 
\begin{equation*}
\prod_{i =1, 2, 3} (Cone\ A_{3i}, S_{3i})
\end{equation*}
in such a way that it is closed. 
Define $Cone\ A \in CONE(H_0)$ by
\begin{equation*}
Cone\ A := Cone\ A_{31}\cup A_{32} \cup A_{33} 
\end{equation*}
(Figure \ref{figure11} (b) right). Then,
\begin{equation}
\Psi_0 \neq \Psi(Cone\  A, M_3)
\end{equation}
(Figure \ref{figure11} (b) right). Therefore, $U_3$ is singular. In this case, $Cone\  A$ dose not give any loop decomposition of $|M_3 \cup M_{U_3}|$. The flow $\Psi_0$ gives an example of the \textit{Penrose stairs} (see Escher’s lithograph ``Ascending and Descending.’’). 
\end{example}

\section{Conclusion}\label{sec4}

\begin{figure}
\centering
\captionsetup{width=1.0\linewidth}
\includegraphics{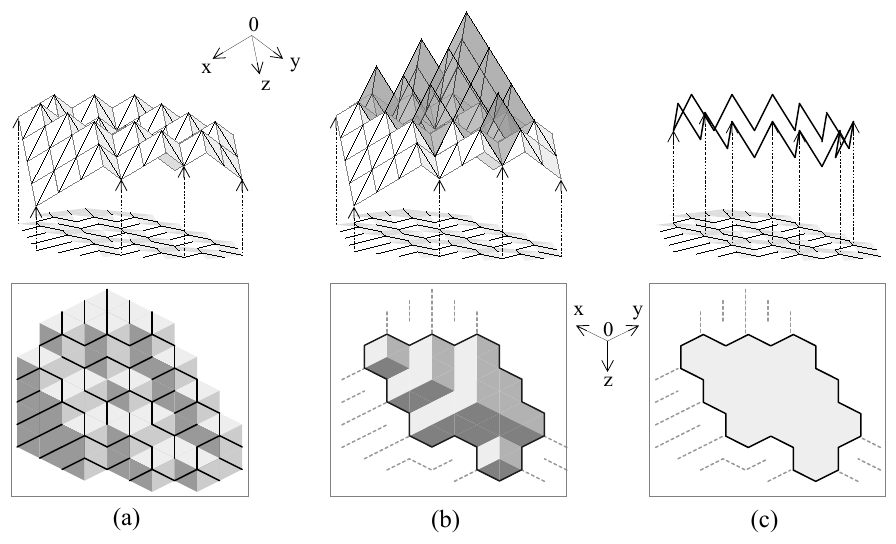}
\caption{(a) Tangent cone. (b) Cotangent cone. (c) The intersection of the two cones.}
\label{figure12}
\end{figure}

In this paper, I considered two questions:
\begin{enumerate}
\item How can we define the shape of a molecule?
\item Is the given loop a molecule?
\end{enumerate}
and proposed two equations to answer these questions:
\begin{enumerate}
\item A defining system of equations for the shape of a molecule, 
\item An equation that determines whether holes within a loop are singular.
\end{enumerate} 
However, we could hardly say that these equations answer the questions.

Regarding the first question, it is unclear how well the system of equations specifies the shape of a molecule. Moreover, some of the permutations of the variables may conserve the shape, but the relation between the permutations and the symmetry of the shape is not yet understood. Furthermore, there is currently no algorithm for solving the system of equations.

Regarding the second question, even if a given loop does not form a molecule, another loop of the same shape may make the region a molecule. If all regular flows are affine (Claim \ref{claim3}), we can compute all regular flows on a region to find a loop without singurarity by placing and/or removing unit cubes (i.e., cube stacking). However, that does not characterize the conditions for the existence of loops without singurarity.

Then, why are we considering these questions in the first place? Here are the motivations behind the research:
\begin{enumerate}
\item We consider the first question to make computers understand shapes (i.e., the ``semantics of shapes''). Our approach defines shapes in a similar way to defining rational numbers. This could allow computers to crunch shapes as well as numbers.
\item We consider the second question to obtain a quantized version of a mathematical theory for studying quantum phenomena in the natural sciences \cite{NM1}. In our case, protein interactions are modeled using a quantized version of differential geometry. There, tangent cones give a quantized version of regular analytic functions.
\end{enumerate}
I believe these motivations give us sufficient reasons for further research.

One of the questions that directly arise from this study is:
Given a system of equations that define the shape of a \underline{molecule complex}, determine whether the system of equations has a solution consisting of a single molecule. If you can answer this question, it will be another answer to the second question above.

Finally, I would like to introduce "cotangent cones" as one of the future research topics. Recall that tangent cones are defined using a lattice spaned by three vectors
\begin{equation*}
(1,\ 0,\ 0), (0,\ 1,\ 0), \text{ and } (0,\ 0,\ 1).
\end{equation*}
Figure \ref{figure12} (a) is a tangent cone and the flow defined by the cone (that defines a molecule). On the other hand, ``cotangent cones'' are defined using another lattice defined by
\begin{equation*}
(0,\ 1,\ 1), (1,\ 0,\ 1), \text{ and } (1,\ 1,\ 0) 
\end{equation*}
Figure \ref{figure12} (b) is a cotangent cone. In this example, by taking the intersection of the two cones, we obtain the boundary of the molecule defined by the tangent cone (Figure \ref{figure12} (c)). That is, we can examine the shapes of molecules witout projecting the ``surface flows'' on a tangent cone onto the hyperplane $H$ (Definition \ref{def_H0}). In general, tangent cones and cotangent cones have more information than the projected image on $H$. However, I do not know how to handle these objects to obtain information about the shapes of molecules and the singularities of loops.

I hope this paper will contribute to the fields of research at the intersection of mathematics and protein science.

\end{document}